\documentclass[11pt,twoside, reqno]{article}
\usepackage{amsmath,amssymb,amsfonts,amsthm,mathtools}
\usepackage{amsfonts}
\usepackage{amssymb}
\usepackage{amsmath}
\usepackage{amsthm}
\usepackage{latexsym}
\usepackage{mathrsfs}
\usepackage{bbm}
\usepackage{comment}

\usepackage{txfonts}
\usepackage{enumerate}
\usepackage{hyperref}
\usepackage{graphicx}
\usepackage{authblk}

\usepackage{color}
\usepackage{pgf,tikz}
\usepackage{mathrsfs}
\usetikzlibrary{arrows}
\usepackage{cite}

\allowdisplaybreaks

\def\fz{\infty}

\def\r{\right}
\def\lf{\left}

\newcommand\SM{\mathscr M}
\newcommand\mstrut{{\phantom{.}}}

\def\wz{\widetilde}

\def\ls{\lesssim}

\def \ptr{/\!/}

\def\scr{\mathscr} \def\rr{\rho}

\def\dsum{\displaystyle\sum}
\def\dint{\displaystyle\int}

\def\dsup{\displaystyle\sup}

\newcommand\newdot{{\kern.8pt\cdot\kern.8pt}}
\newcommand\smallbullet{{\displaystyle\boldsymbol{\cdot}}}
\newcommand\bolddot{{\displaystyle\boldsymbol{.}}}

\def\e{{\rm e}}
\def\mathpal#1{\mathop{\mathchoice{\text{\rm #1}}%
   {\text{\rm #1}}{\text{\rm #1}}%
   {\text{\rm #1}}}\nolimits}

\newcommand\Ric{\mathpal{Ric}}
\newcommand\Hess{\mathpal{Hess}}
\newcommand\Riem{\mathpal{Riem}}

\newcommand\End{\mathpal{End}}
\newcommand\Hom{\mathpal{Hom}}

\newcommand\tr{\mathpal{tr}}
\newcommand\loc{\mathpal{loc}}
\newcommand\vol{{\operatorname{vol}}}

\newcommand\E{\mathbb{E}}
\newcommand\R{\mathbb{R}}
\newcommand\N{\mathbb{N}}
\newcommand\Z{\mathbb{Z}}

\newtheorem{theorem}{Theorem}[section]
\newtheorem{lemma}[theorem]{Lemma}
\newtheorem{corollary}[theorem]{Corollary}
\newtheorem{proposition}[theorem]{Proposition}
\theoremstyle{definition}
\newtheorem{remark}[theorem]{Remark}

\newtheorem{definition}[theorem]{Definition}

\def\supp{{\mathop\mathrm{\,supp\,}}}

\def\d{\mathrm{d}}
\def\ss {\sqrt}
\numberwithin{equation}{section}
\newcommand\1{\mathbbm{1}}

\def\supp{{\mathop\mathrm{\,supp\,}}}

\def\gg{\gamma} \def\vv{\varepsilon}
 \def\T{{\bf T}} \def\si{\sigma} \def\nn{\nabla} \def\OO{\Omega}\def\DD{\Delta} \def\dd{\delta} \def\ff{\frac}\def\beg{\begin} \def\aa{\alpha}\def\d{{\rm d}}
\def\tt{\tilde} \def\beq{\beg{equation}}

\begin{document}

\title{\vskip-2.1cm\bf\large $L^p$-Boundedness of the Covariant Riesz Transform\\ on Differential Forms for $p>2$
  \footnotetext{\hspace{-0.35cm} 2010 {\it Mathematics Subject
      Classification}. Primary: 35K08; Secondary: 58J65, 35J10, 47G40.
    \endgraf {\it Keywords and phrases}.  Covariant Riesz transform,  Heat kernel, Bochner formula,
    Calder\'{o}n-Zygmund inequality,
    Hardy-Littlewood maximal function,
    Kato inequality.  \endgraf This work has been supported 
    in part by the National Key R\&D Program of China (2022YFA1006000), NSFC (12531007)   and Natural Science Foundation of Zhejiang Provincial (Grant No.\,LGJ22A010001). 
}}

\author[1]{Li-Juan Cheng}
\author[2]{Anton Thalmaier}
\author[3]{Feng-Yu Wang}

\affil[1]{\scriptsize School of Mathematics, Hangzhou Normal
  University,\par
  Hangzhou 311121, People's Republic of China\par
  \texttt{lijuan.cheng@hznu.edu.cn}\vspace{1em}}

 \affil[2]{\scriptsize Department of Mathematics, University of Luxembourg,
  Maison du Nombre,\par
  L-4364 Esch-sur-Alzette, Luxembourg\par
  \texttt{anton.thalmaier@uni.lu}\vspace{1em}}

  \affil[3]{\scriptsize Center for Applied Mathematics and KL-AAGDM, Tianjin
      University,\par Tianjin 300072, People's Republic of China\par
  \texttt{wangfy@tju.edu.cn}}

\date{\small\today}
\maketitle

\begin{abstract} {\noindent 
We establish the \(L^p\)-boundedness, for \(p>2\), of the covariant Riesz transform
\[
\nabla(\Delta_\mu^{(k)}+\sigma)^{-1/2}
\]
on differential forms over a class of complete weighted Riemannian manifolds. The proof is based on an  heat-kernel criterion involving local volume doubling, heat kernel upper estimates, Kato-type curvature control, and gradient bounds for the heat semigroup on forms. Under curvature-dimension assumptions and Kato-type curvature bounds, this criterion applies and yields boundedness for all sufficiently large \(\sigma\). In particular, in the unweighted case, the result confirms a conjecture of Baumgarth, Devyver and Güneysu~\cite{BDG-23}. As an application, we obtain Calderón--Zygmund inequalities for \(p>2\) on weighted manifolds, which extends  the recent work \cite{CCT} on manifolds without weight.
}
\end{abstract}

\tableofcontents

\section{Introduction\label{s1}}
Let \((M, g)\) be a complete geodesically connected \(m\)-dimensional Riemannian manifold, \(\nabla\) the Levi-Civita covariant derivative, and \(\Delta\) the Laplace-Beltrami operator. The operator \(\Delta\) is understood as a self-adjoint positive operator on \(L^2(M)\). The Riesz transform on the space of smooth functions on Euclidean space, defined by \(\T^{(0)} := \nabla \Delta^{-1/2}\), was first introduced by Strichartz \cite{Str83}. The \(L^p\)-boundedness of \(\T^{(0)}\) and its extension to manifolds have been the subject of extensive research; see \cite{Chen92,Barkry85I,Barkry85II,Bakry87,PTTS-2004,TX-99,CD-01,CD-03,Li91,ThW-04} and the references therein.

In this paper, we investigate the \(L^p\)-boundedness of the covariant Riesz transform on the space \(\OO^{(k)} := \Gamma(\Lambda^k T^*M)\) of smooth differential \(k\)-forms for \(k \in \{1, \ldots, m\}\):
\begin{align}\label{def-T}
\T^{(k)}_\si := \nabla (\Delta^{(k)} + \sigma)^{-1/2}, \quad\si\in (0,\infty),
\end{align} 
where  \(\nabla\) denotes the Levi-Civita covariant derivative, and  \(\Delta^{(k)}\) the Hodge Laplacian on $\OO^{(k)}$.

For \(p \in (1,2)\), the \(L^p\)-boundedness of $\T_{\sigma}^{(k)}$ was established by F.-Y. Wang and A. Thalmaier \cite{ThW-04}, following the approach of Coulhon and Duong \cite{TX-99} by verifying the doubling volume property, Li–Yau heat kernel upper bounds, and heat kernel derivative estimates. This result was later improved by Baumgarth, Devyver, and G\"uneysu \cite{BDG-23}, who relaxed the boundedness condition on the derivatives of the curvature, and further in \cite{CTW}, where the curvature derivative condition was entirely removed.
However,  as explained in \cite{ThW-04}, the  argument developed in \cite{TX-99} does not apply to the case $p>2$.  The \(L^p\)-boundedness of $\T_{\sigma}^{(k)}$  in this regime remained an open problem for some time and was formulated as a conjecture by Baumgarth, Devyver, and G\"uneysu \cite{BDG-23}.

\paragraph{Conjecture \cite{BDG-23}.}  \emph{Assume that the Riemannian curvature tensor $\Riem$ satisfies
 $$ \max\left\{\|\Riem\|_{\infty}, \|\nabla \Riem\|_{\infty} \right\}\leqslant A $$
for some constant $A$. Then  there exists a constant $\si_0 \in (0,\infty)$  depending only on $A$ and $m$, such that  for any $\si \in [\si_0,\infty)$ and $p\in (1,\infty),$
$$\sup_{1\leqslant k\leqslant m} \big\|\T_\si^{(k)}\big\|_{p\rightarrow p}\leqslant B$$ holds for some  constant $B\in (0,\infty)$ depending only on $A$,
$\si$ and $m$, where $\|\cdot\|_{p}$ denotes the $L^p$-norm on $M$ with respect to the volume measure.   }

\medskip
We note that  when $\nn$ is replaced by the exterior differential  \(\d^{(k)}\) or its
\(L^2\)-adjoint  \(\delta^{(k)}\), the \(L^p\)-boundedness of  \(\d^{(k)} (\Delta^{(k)}+\sigma)^{-1/2}\) and \(\delta^{(k-1)}(\Delta^{(k)}+\sigma)^{-1/2}\)   has been derived in  \cite{Bakry87, Li2010}, but the techniques developed therein do not apply to the covariant Riesz transform $\T_{\sigma}^{(k)}$.

The main goal of this paper is to confirm the above conjecture by proving the \(L^p\)-boundedness of $\T_{\sigma}^{(k)}$ for \(p \in (2, \infty)\), since the case \(1 < p \leqslant  2\) has already been settled in \cite{CTW}.
According to G\"uneysu and Pigola \cite{GP-15}, the \(L^p\)-boundedness of $\T^{(1)}_\si$  and $\T^{(0)}_\si$ implies that of  \(\Hess (\Delta + \sigma)^{-1}\), since
\[
\mathrm{Hess}(\Delta + \sigma)^{-1} = \nabla (\Delta^{(1)} + \sigma)^{-1/2} \circ \mathrm{d} (\Delta + \sigma)^{-1/2}.
\]
The \(L^p\)-boundedness of \(\mathrm{Hess}(\Delta + \sigma)^{-1}\), known as the Calderón–Zygmund inequality, was recently established for \(p > 2\) by Cao, Cheng, and Thalmaier \cite{CCT}. This  provides positive evidence  for the conjecture when $k=1$.

In this paper, under certain curvature conditions, we establish the $L^p(\mu)$-boundedness of the covariant Riesz transform on the space $\OO^{(k)}$  over a weighted Riemannian manifold:
$$\T^{(k)}_{\mu,\si}:= \nn (\DD_\mu^{(k)}+\si)^{-1/2},\quad 1 \leqslant k \leqslant m,$$
where $\mu(\d x):= \e^{h(x)}\,\vol(\d x)$ for some $h\in C^2(M)$ and the  volume measure $\vol$. The weighted Hodge Laplacian is defined as
\begin{equation} \label{DK} \DD_\mu^{(k)}:= \dd_\mu^{(k+1)} \d^{(k)}+ \d^{(k-1)} \dd_\mu^{(k)} \end{equation}
 with  $\dd_\mu^{(k+1)}\colon \OO^{(k+1)}\to \OO^{(k)}$ being the $L^2(\mu)$-adjoint of $\d^{(k)}$.  In particular, when $h=0$, we have $\mu(\d x)=\vol(\d x)$ and $\T_{\mu,\si}^{(k)}=\T^{(k)}_\si$, thereby confirming 
 the above conjecture.  For $k=0$ we write $\d=\d^{(0)}$ and
 $$\DD_\mu=\DD_\mu^{(0)} := \dd_\mu^{(1)}\d= \DD-\nn h,$$
 where $\DD$ is the Laplacian on $M$.

The remainder of the paper is organized as follows. In Section 2 we present our main results and their consequences. The proofs are given in Section 3 and Section 4, respectively.

\smallskip{\bf Acknowledgements.} The authors are indebted to Batu
G\"uneysu, Stefano Pigola and Giona Veronelli for helpful
comments on the topics of this paper.\goodbreak

\section{Main results and consequences}
We first introduce a general criterion on the $L^p$-boundedness ($p>2)$ of $\T_{\mu,\si}^{(k)}$ in terms of estimates on heat kernels and their gradients. Then we verify this criterion by exploiting curvature conditions, which in turn
   provides a positive answer to {\bf Conjecture \cite{BDG-23}.} As a consequence, the Calder\'on-Zygmund inequality is presented for $ p>2.$

   Let $P_t$ be the diffusion semigroup on $M$ generated by the weighted Laplacian $-\DD+\nn h$, and $p_t$ be the heat kernel of $P_t$ with respect to $\mu$.
  We introduce below the \textit{contractive Dynkin  class} of functions, which is also called generalized or extended Kato class, and has been systematically studied first by P. Stollmann and J. Voigt in \cite{Peter-Jurgen}.

  \begin{definition}\label{def-K}{(Contractive Dynkin  class)}
    We say that a function $f$ on $M$ belongs to the class $\hat{\mathcal{K}}$
    (in short: $f\in \hat{\mathcal{K}}$) if
\begin{align*}
\lim_{\aa\downarrow 0} \sup_{x\in M} \int_M \int_0^{\alpha} p_s(x,y) |f(y)| \,\d s\, \mu(\d y)<1.
\end{align*}
\end{definition}

Note that $\hat{\mathcal{K}}$ contains the usual Kato class $\mathcal{K}$,
defined as the set of functions $f$ such that
\begin{align*}
\lim_{\alpha\downarrow  0} \sup_{x\in M} \int _M \int_0^{\alpha} p_s(x,y)
|f(y)| \, \d s\, \d\mu(y)=0.
\end{align*}
The Kato class plays an important role in the study of Schr\"{o}dinger
operators and their semigroups, see Simon \cite{Simon} and the reference therein.
 It is straight-forward that $f\in \hat{\mathcal{K}}$ if $f$ is bounded.

To state the main result, we first introduce the weighted volume on $M$ and the weighted curvature operator  on $\OO^{(k)}$.
 For $x\in M$ and $r>0$, let $B(x,r)$ be the open geodesic ball centered at $x$ of radius~$r$, and
$$\mu(x,r):=\mu\big(B(x,r)\big)=\int_{B(x,r)}\e^{h(y)}\,\vol(\d y).$$
The weighted curvature operator $\scr R^{(k)}_h$ on $\OO^{(k)}$ is defined as
 $$ \mathscr{R}^{(k)}_h(\eta):= \scr R^{(k)}(\eta) -(\Hess h)^{(k)}(\eta),$$
 where for an  orthonormal frame $\left(e_{i}\right)_{1\leqslant i \leqslant  m}\in O(M)$ with  dual frame $(\theta^j)_{1\leqslant j\leqslant  m},$
 \beg{align*}
&\scr R^{(k)} :=-\sum_{i, j=1}^{m} \theta^{j} \wedge\big(e_{i}\mathop{\lrcorner} \mathrm{R}(e_{j}, e_{i})\big),\\
&(\Hess h)^{(k)}:= \sum_{i,j=1}^m e_i\big(e_j(h)\big)\Big( \theta^j\wedge \left(e_i \mathop{\lrcorner} \cdot \right) \Big),\\
&X\mathop{\lrcorner} \eta\, (X_1,\ldots, X_{k-1}):= \eta(X,X_1,\ldots, X_{k-1}),\ \ \eta\in \OO^{(k)},\  X, X_1,\ldots, X_{k-1}\in TM.\end{align*}
   When $k=1$, we have
$$\mathscr{R}^{(1)}_h =\Ric_h:=  \Ric-\Hess h,$$ where $\Ric$ is the Ricci curvature of $M$.
By the Weitzenb\"{o}ck formula, we have the decomposition
$$\DD_\mu^{(k)}= \square_{\mu} +\mathscr{R}^{(k)}_h,$$
with respect to the Bochner Laplacian $\square_{\mu}:= \nn_\mu^* \nn,$ where
  $\nn_\mu^*$ denotes the $L^2(\mu)$-adjoint operator of $\nn$.

  Moreover, let $R^{(k)}$ be the curvature tensor on $\OO^{(k)}$. For any $\eta \in \OO^{(k)} $ and $v \in TM$,  define
   \begin{align*}& \big(R^{(k)}\cdot\eta\big)(v):=\sum_{i=1}^{n}R^{(k)}(v,e_i)\eta(e_i),\\
 & (\nabla \cdot R^{(k)})(v)\eta:=  \sum_{i=1}^n (\nabla_{e_i} R^{(k)})(e_i, v)\eta, \\
& \big(R^{(k)}(\nabla h)\big) (v)\eta :=R^{(k)}(v,\, \nabla h)\eta.
\end{align*}

For any $1\leqslant k\leqslant m$, let
$$P_t^{(k)}:=\e^{-t \DD_\mu^{(k)}},\quad t\geqslant 0$$
be the semigroup  on $\OO^{(k)}$ generated by $-\DD_\mu^{(k)}$ with $\DD_\mu^{(k)}$
defined in \eqref{DK}.
Finally, denote by $\OO^{(k)}_{b,1}$ the class of differential forms $\eta\in \OO^{(k)}$ for which $|\eta|+|\nn\eta|$ is bounded.

We are going to prove $L^p(\mu)$ boundedness of  $\T_{\mu,\si}^{(k)} $ for $p>2$ under the following assumptions.

\begin{enumerate}[{\bfseries (A)}]
\item There exist a constant $A \in (0,\infty)$ and a positive function $V_k\in \hat{\mathcal{K}}$ such that the following conditions hold:
  \begin{align}
  &\mu(x, \alpha r) \leqslant A \mu(x,r) \,\alpha^{m}\exp (A(\alpha -1)r),
  \quad x\in M,\ \aa>1,\ r>0,
  \label{A1}  \tag{{\bf LD}}
  \\[1.5ex]
  &p_t( x, x)\leqslant \frac{A\e^{A t}}{\mu(x, \sqrt{t})},
  \quad x\in M,\ t>0,
  \label{A2} \tag{{\bf UE}}
  \\[1.5ex]
  & \big\langle \mathscr{R}^{(k)}_h(\eta), \eta \big\rangle \geqslant -V_k|\eta|^2,
  \quad\eta\in \OO^{(k)},
  \label{BD-R}  \tag{{\bf Kato}}
  \\[1.5ex]
  &|\nn P_t^{(k)} \eta|\leqslant \e^{A+At} \min\Big\{t^{-1/2}  (P_t|\eta|^2)^{1/2},\
 \big((P_t|\nn\eta|^2)^{1/2}+  (P_t|\eta|^2)^{1/2}\big)\Big\},\quad\eta\in \OO^{(k)}_{b,1},\ t>0.
  \label{A3}\tag{{\bf GE}}
  \end{align}
\end{enumerate}

\begin{theorem}\label{main-them1} Assume that {\bf (A)} holds for $k\in \N$.
Then there exists a constant $\si_0\in (0,\infty)$ depending only on $A$    such that for any $p\in (2,\infty),$
\beq\label{A} \sup_{\sigma\in [\si_0,\infty)}\|\T_{\mu,\si}^{(k)}\|_{p\rightarrow p}\leqslant B\end{equation}
 holds for some constant $B\in (0,\infty)$ depending on $p,\,  m,  A$ and the Kato date of $V_k$.\end{theorem}

  For the convenience of applications, we   present  below explicit  curvature conditions which ensure hypothesis {\bf (A)}. To this end, for $f\in C^\infty(M)$, let
$\Gamma_{2}(f, f):=-\frac{1}{2} \Delta_{\mu}|\nabla f|^{2}+(\nabla \Delta_{\mu} f, \nabla f)_g$.
\beg{enumerate}\item[{\bfseries (C)}]   There exist  constants $N\geqslant m$ and  $K_0>0$, and function $K\in {\mathcal{K}}$  such that
\begin{align}
 &\Gamma_{2}(f, f)\geqslant-K_0 |\nabla f|^{2}+\frac{1}{N}(\Delta_{\mu} f)^{2},\quad f\in C^\infty(M),\label{CD}\\
 & \big|\scr R_h^{(k)}\big| + \big|R^{(k)} \smallbullet\, \big| +  \big| \nabla  \cdot R^{(k)}+R^{(k)}(\nabla h) + \nabla \mathscr{R} ^{(k)}_h \big| \leqslant K.\hskip3cm  \label{CV}
\end{align}
 \end{enumerate}

The next theorem is then a consequence of Theorem \ref{main-them1}.

\begin{theorem}\label{C1} Assume that {\bf (C)} holds for $k\in \N$. Then there exists a  constant $\si_0 \in (0,\infty)$ depending only on $K_0$, $K$ and $N$ such that for any $p\in (2,\infty)$,
 $$\sup_{\si\in [\si_0,\infty)}\big\|\T_{\mu,\si}^{(k)}\big\|_{p\rightarrow p}\leqslant B$$
 holds for some constant $B\in (0,\infty)$ depending on constants $p, K_0, m, N$ and the Kato data of $K$.

 In particular, if {\bf (C)}  holds for $k=1$, then there exists a
 constant $\si_0 \in (0,\infty)$ depending only on $m$, $K_0$, $N$ and the Kato data of $K$, such that
 $\|\Hess \, (\DD_{\mu} +\si)^{-1}\|_{L^p(\mu)}<\infty$ 
 for all \(\sigma \geqslant \sigma_0\) and \(p > 2\). As a consequence, the
 Calderón–Zygmund inequality holds for some constant \(C \in (0, \infty)\):
 \begin{equation}\label{CZ}
    \big\|\Hess  f \big\|_{L^p(\mu)}\leqslant C \left(\|f\|_{L^p(\mu)}+ \|\Delta_{\mu} f\|_{L^p(\mu)}\right),\quad f\in C_0^\infty(M). 
  \end{equation}
 \end{theorem}

Note that on a geodesically complete manifold with Riemann curvature tensor $\Riem$
satisfying $\|\Riem\|_\infty<\infty$, there exists a sequence of Hessian cut-off functions (see \cite{GP-15}, p.~362), such that inequality \eqref{CZ} extends from $C_0^\infty(M)$ to
$f\in C^\infty(M)\cap L^p(\mu)$ with $\|\Delta_\mu f\|_\infty<\infty$.

\section{Proof of the Main Theorem}

To prove our main result (Theorem \ref{main-them1}), we
shall need the following  lemma,  which is due to \cite{TX-99}.

\begin{lemma}[\cite{TX-99}]\label{lem3}
If  \eqref{A1} holds, then   there exist a   constant  $c\in (0,\infty)$ and a function $C\colon (0,\infty)\to (0,\infty)$ depending only on $A$ and $m$, such that
\begin{align}\label{eqn-rho}
\int_{B\lf(x,\sqrt t\r)^c} \mathrm{e}^{-{\gg\rho^{2}(x, y)}/{s}} \,\mu(\d y) \leqslant C_{\gamma}\, \mu(x, \sqrt{s}) \,\mathrm{e}^{{cs}/\gamma  -{\gamma t}/ s},\quad s,t,\gamma>0,\ x\in M,
\end{align}
where $B(x,\sqrt t)^c:=\lf\{y\in M: \rr(x,y)\geqslant \sqrt t\r\}.$  In particular, $t\to 0$ yields
\begin{align}\label{integ-esti}
\int_M \frac{\e^{-{cs}/{\gamma} }}{C_{\gamma}\, \mu(x,  \sqrt{s})} \e^{-{\gg\rho^2(x,y)}/{s}}\, \mu(\d y) \leqslant 1,\quad s, \, \gamma>0,\ x\in M.
\end{align}
\end{lemma}

\subsection{Heat kernel estimates}

By the usual abuse of notation, the  corresponding self-adjoint realizations
of $\Delta_{\mu}$ and $\Delta^{(k)}_{\mu}$ will again
be denoted by the same symbol. By local parabolic regularity, for all square-integrable $k$-forms $a \in L^2(\OO^{(k)},\, \mu)$, the time-dependent $k$-form
$$
(0,\infty) \times M \ni(t, x) \mapsto P_t^{(k)} a (x) \in\OO_x^{(k)}:= \Lambda^k T_x^*M
$$
has a smooth representative which extends smoothly to $[0,\infty)\times M$ if $a$ is smooth.
In addition, there exists a unique smooth heat kernel $p_t^{(k)}$ to $P_t^{k}$ with respect to the measure $\mu$, understood as a map
\begin{align*}
(0,\infty)\times M \times M \ni (t,x,y)\mapsto p_t^{(k)}(x,y)\in\Hom(\OO^{(k)}_y,\,  \OO^{(k)}_x )
\end{align*}
such that
\begin{align*}
P_t^{(k)}a(x)=\int_M p_t^{(k)}(x,y)a(y)\,\mu(\d y).
\end{align*}

Let $P_t^{V_k}$ be the heat semigroup associated to $\Delta_{\mu}+V_k$ and $p^{V_k}_{t}(x, y)$ be the
corresponding heat kernel.
If condition (\textbf{Kato}) in \textbf{(A)} holds,
then
\begin{align*}
| p_t^{(k)}(x,y)|\leqslant p^{V_k}_{t}(x, y).
\end{align*}
Combining this inequality with  \cite[Lemma 2.2]{MO2020} for  the upper bound estimate on $p^{V_k}_{t}(x, y)$, we obtain the following result; see  \cite{CZ07,  Sturm, Zhang00, Zhang01} for earlier results on Sch\"{o}dinger heat kernel estimates.

\begin{lemma}\label{CZ07}
  \label{Gaussian-estimate}
  Let $M$ be a complete non-compact Riemannian manifold satisfying  \eqref{A1}, \eqref{A2} and \eqref{BD-R}.   There exists a function \(C \colon (0, {1}/{4}) \to (0, \infty)\), depending only on \(A\), \(m\), and \(V_k\), such that for all \(x, y \in M\), \(t > 0\), and \(\gamma \in (0, 1/4)\),
\begin{align}\label{HE}
\big|p_t^{(k)}(x, y)\big| \leqslant \frac{C_\gg \e^{C_\gg t}}{\mu(y, \sqrt{t})} \exp \left(-\ff{\gg \rho(x, y)^{2} } t\right),\quad \forall x,y\in M, \, t>0,\ 0<\gg< 1/4,
\end{align}
where we write \(C_\gamma = C(\gamma)\) for notational simplicity.  
This estimate, combined with \eqref{integ-esti}, yields 
\begin{align}\label{integ-esti2}
\sup_{t\in (0,1],\,x\in M} \int_M   \big|p_t^{(k)}(x, y)\big| \, \mu(\d y) <\infty.
\end{align}
\end{lemma}
 We are now ready to  present the following estimate.

\begin{theorem}\label{esti-L2}
Let $M$ be a complete non-compact Riemannian manifold satisfying  the condition {\bf (A)}.  There exists $C\colon (0,1/4)\to (0,\infty)$, depending only on $A,m$ and $V_k$, such that
\begin{align*}
\int_M \lf(t |\nabla p_t^{(k)}(z,y)|^2+|p_t^{(k)}(z,y)|^2\r) \e^{{2\gg\rho^2(z,y)}/{t} }\, \mu(\d z) \leqslant \frac{C_\gg \e^{C_\gg t }}{\mu(y,\sqrt{t})},\quad y\in M,\ t>0,\ 0<\gg< 1/4.
\end{align*}
\end{theorem}

\begin{proof}
By \cite[Lemma 2.2]{MO2020}, if \( V_k \in \hat{\mathcal{K}} \), then there exist constants \( \kappa \in [0,1) \) and \( c_1 > 0 \), depending only on \( V_k \), such that
\begin{align}\label{V-ineq}
\int_M V_k |f|^2\, \d \mu \leqslant \kappa\left\| \, |\nabla f |\, \right\|_2^2+c_1\| f \|_2^2
\end{align}
for all \( f \in W^{1,2}(M) \). It means in particular that the operator \( \Delta - V_k + c_1 \) is strongly positive. Combining this with the Gaussian upper bound \eqref{HE} in Lemma \ref{Gaussian-estimate}, we find that the proof of \cite[Theorem 2.6]{CTW} remains valid under the present assumptions. As a consequence,
\begin{align*}
\int_M t |\nabla p_t^{(k)}(z,y)|^2 \e^{{2\gg\rho^2(z,y)}/{t} }\, \mu(\d z) \leqslant \frac{\tt C_\gg \e^{\tt C_\gg  t }}{\mu(y,\sqrt{t})},\quad y\in M,\ t>0,\ \gg\in (0,1/4)
\end{align*}
for some $\tt C\colon (0,1/4)\to (0,\infty)$ depending only on $A,m$ and $V_k$.
Combined with \cite[Lemma 2.5]{CTW}, this yields the desired estimate for some function \( C \colon (0, 1/4) \to (0, \infty) \).
\end{proof}

The following is a direct consequence of Theorem \ref{esti-L2} and extends \cite[Theorem 1.2]{BDG-23} to the case of weighted manifolds.

\begin{corollary}\label{gradient-estimate}
 Let $M$ be a complete non-compact Riemannian manifold satisfying the  condition {\bf (A)}.   There exists    $C\colon (0,1/8)\to (0,\infty)$ depending only on $A,m$ and the function $V_k$,  such that
\begin{align}\label{EstGradHeatKernel}
  |\nabla p_t^{(k)}(\newdot,y)(x)|\leqslant
  \frac{C_\gg \e^{C_\gg t}}{\sqrt{t}\,\mu(y,\sqrt{t})}\exp\Big(-\ff{\gg \rho^2(x,y)}t\Big), \quad
  \forall x,y\in M, \, t>0,\ 0<\gg<1/8.
\end{align}
\end{corollary}
\begin{proof}
Let $x,y\in M$. It is easy to see that
$$\nabla p_{2t}^{(k)}(\newdot, y)(x)= \nabla P_t^{(k)}\lf( p_t^{(k)}(\newdot, y)\r)(x). $$
Using condition \eqref{A3}, we have
$$ |\nabla P_t^{(k)}\eta|\leqslant \frac{\e^{A+At}}{\sqrt{t}} (P_t|\eta|^2) ^{1/2},$$
for $\eta\in \Omega^{(k)}$ with $P_t(|\eta|^2)<\infty$.
 We use this inequality with $\eta(z)=p_t^{(k)}(\newdot, y)(z)$ to obtain
  \begin{align*}
    &\big|\nabla P_t^{(k)}\lf( p_t^{(k)}(\newdot, y)\r)\big|(x)
      \leqslant  \frac{\e^{A+At}}{\sqrt{t}}
      \lf(\int_M p_t(x,z)\bigl| p_t^{(k)}(z,y)\bigr|^2\, \mu(\d z)\r)^{1/2}.
  \end{align*}
By Theorem \ref{esti-L2},  this implies that for any $\gg\in (0,1/4)$,
  \begin{align}\label{Cor2.11}
    |\nabla p_{2t}^{(k)}(\newdot,y)(x)|
    &\leqslant \frac{\e^{A+At}}{\sqrt{t}} \left(\int_M |p_t^{(k)}(z,y)|^2 \e^{  \frac{2\gg \rho^2(z,y)}{t}- \frac{2\gg\rho^2(z,y)}{t}}p_t(x,z)\,\mu(\d z)\right)^{1/2}\nonumber \\
    &\leqslant  \frac{C_\gg \e^{A+(A+C_\gg)t}}{\sqrt{t\mu\big(y,\sqrt{t}\big)}}\sup_{z \in M}\left\{\e^{- \frac{2\gg\rho^2(z,y)}{t}}p_t(x,z)\right\}^{1/2} .
  \end{align}
  Since $p_t(x,x)$ satisfies the diagonal estimate \textbf{(UE)}, from the proof
of \cite[Lemma 3.2]{MO2020},   there exists   a function $\tt C\colon (0,1/4)\to (0,\infty)$ depending only on $A$ and $m$ such that
  \begin{align}\label{Gaussion-upper-bound}
  p_t(x,z)\leqslant  \frac{\tt C_\gg \e^{\tt C_\gg  t}}{\mu(x, \sqrt{t})} \exp \left(-\ff{2\gg \rho(x, z)^{2}} t\right),\ \ 0<\gg<1/8,\ t>0,\ x,y\in M.
  \end{align}
  By \eqref{A1}, there exists a  decreasing function
  $c\colon (0,1) \to (0,\infty)$ depending only on $A$ and $m$ such that
 \beg{align*}
 \mu\Big(y,\ss t\Big)&\leqslant \mu\Big(x, \ss t \big(1+ t^{-1/2}\rr(x,y)\big)\Big) \leqslant A\mu\Big(x,\ss t\Big) \big(1+t^{-1/2}\rr(x,y)\big)^m \e^{A\rr(x,y)}\\
 & \leqslant c_\vv \mu\Big(x,\ss t\Big) \exp\left(\ff{\vv   \rr(x,y)^2}t +c_\vv t\right),\quad \vv \in (0,1),\ t>0,\ x,y\in M.\end{align*}
  Combining this with \eqref{Gaussion-upper-bound}
   and
  $$2 \rr(x,z)^2+2\rr(y,z)^2\geqslant \rr(x,y)^2,$$
  we find $\hat C\colon \big\{(\gg,\vv)\colon\ 0<\vv<\gg<1/8\big\}\to (0,\infty)$ depending only on $A,m$ and $V_k$, such that 
  $$
 \sup_{z \in M}\left\{\e^{- \frac{2\gg\rho^2(z,y)}{t}}p_t(x,z)\right\} \leqslant   \frac{\hat C_{\gg,\vv} \e^{\hat C_{\gg,\vv} t}}{\mu(y,\sqrt{t})}\exp\left(- \ff{(\gg-\vv)\rho^2(x,y)}t\right),\quad x,y\in M,\ t>0,\ 0<\vv<\gg<1/8.
  $$
  Combining this with \eqref{Cor2.11} yields 
   \begin{align*}
    |\nabla p_{2t}^{(k)}(\newdot,y)(x)|
    &\leqslant  \frac{\sqrt{\hat C_{\gg,\vv}}  C_\gg \e^{A+(A+C_\gg)t+\hat C_{\gg,\vv} t/2}}{\sqrt{t}\mu\big(y,\sqrt{t}\big)}\exp\Big(- \ff{(\gg-\vv)\rho^2(x,y)}{2t}\Big) .
  \end{align*}
  By this and \eqref{A1}, we obtain
  the desired estimate for some $C\colon (0,1/8)\to (0,\infty).$
\end{proof}

As a consequence of the pointwise estimates in Corollary \ref{gradient-estimate} and the local volume
doubling property \eqref{A1}, we have the following result which  extends \cite[Corollary 1.3]{BDG-23} to  the case of a weighted $L ^p$-estimates
of $|\nabla p_t^{(k)}|$.

\begin{theorem}\label{Lp-estimate}
  Let $M$ be a complete non-compact Riemannian manifold satisfying
  condition {\bf (A)}. Then    for any $p\in [1,\infty)$  there exists a function  $C\colon (0,1/8)\to (0,\infty)$ depending only on $p, A,m$ and the function $V_k$, such that
$$
\int_M\left|\sqrt{t} \nabla p_t^{(k)}(x, y)\right|^p \e^{{\gamma p \rho^2(x,y)}/{t}}\, \mu(\d x) \leqslant \frac{C_\gg \e^{C_\gg t}}{\lf(\mu\big(y, \sqrt{t}\big)\r)^{p-1}},\quad  y \in M, \ t>0,\ 0<\gg<1/8.
$$

\end{theorem}

\begin{proof}
  According to inequality \eqref{integ-esti}, we find a function
  $h\colon(0,\infty)\to (0,\infty)$ depending only on $A,m$, such that
$$
\int_{M} \mathrm{e}^{- {\gg\rho^{2}(x, y)}/{t}} \mu(\d x) \leqslant h_\gg\,\mu\big(y,\sqrt{t} \big )\,\mathrm{e}^{h_\gg t}, \quad t,\gg>0.
$$
By  Corollary \ref{gradient-estimate}, there exists $C\colon(0,1/8)\to (0,\infty)$ depending on $ A,m$ and $V_k$ such that 
\begin{align*}
&\int_M\left|\sqrt{t}\nabla p_t^{(k)}(x, y)\right|^p \e^{{(1-\vv)\gamma p \rho^2(x,y)}/{t}}\, \mu(\d x) 
\leqslant  \frac{C^p_{\gg} \e^{ pC_{\gg} t}}{\mu(y, \sqrt{t})^p}\int_M \e^{{-\left(p\gg- p (1-\vv)\gamma\right) \rho^2(x,y)}/{t}}\, \mu(\d x).
\end{align*}
Then by Lemma \ref{lem3}, we find $C,c\colon (0,1/8)\times (0,1)\to (0,\infty)$ depending only on $p,A,m$ and $V_k$, such that
\begin{align*}
&\int_M\left|\sqrt{t}\nabla p_t^{(k)}(x, y)\right|^p \e^{{(1-\vv)\gamma p \rho^2(x,y)}/{t}}\, \mu(\d x) \leqslant   \frac{C_{\gg,\vv} \e^{ c_{\gg,\vv} t}}{\mu(y, \sqrt{t})^{p-1}},\ \ \ (\gg,\vv)\in (0,1/8)\times (0,1),\ t>0,
\end{align*}
which completes the proof.
\end{proof}

We now introduce $L^2$-Davies-Gaffney bounds under condition ${\bf (A)}$
which extend the $L^2$-Davies-Gaffney bound in \cite[Theorem 1.9]{BDG-23}.
Recall that for two non-empty subsets \(E,F\subset M\), their separation is defined by
\[
d(E,F):=\inf\{\rho(x,y):x\in E,\ y\in F\}.
\]

\begin{lemma}\label{lem-Gaffney}
Assume that \eqref{A1}, \eqref{A2} and \eqref{BD-R} hold. Then there exist constants
\(c_1,c_2>0\), depending only on \(A,m\) and on the data of  \(V_k\), such that
for all non-empty relatively compact subsets \(E,F\subset M\), all \(t>0\), and all
\(\alpha\in L^2(\Omega^{(k)},\mu)\) with \(\operatorname{supp}\alpha\subset E\),
\begin{align}\label{Gaffney-ineq}
\|1_F\sqrt t\,|\nabla P_t^{(k)}\alpha|\|_2
\leqslant
c_1(1+\sqrt t)
\exp\left(-c_2\frac{d(E,F)^2}{t}\right)
\|1_E|\alpha|\|_2 .
\end{align}
\end{lemma}

\begin{proof}
All constants below depend only on \(A,m\) and on the Kato data of \(V_k\).
By Lemma \ref{Gaussian-estimate} and the standard Davies--Gaffney argument, one has
\[
\|1_FP_t^{(k)}\alpha\|_2
\leqslant
C \e^{-c d(E,F)^2/t}\|1_E\alpha\|_2.
\]
Moreover, since \(P_t^{(k)}\) is a self-adjoint analytic semigroup on
\(L^2(\Omega^{(k)},\mu)\), the Cauchy integral argument yields
\[
t\|1_F\Delta_\mu^{(k)}P_t^{(k)}\alpha\|_2
\leqslant
C \e^{-c d(E,F)^2/t}\|1_E\alpha\|_2.
\]
Equivalently,
\begin{align}\label{L2-Gaffney}
\|1_FP_t^{(k)}\alpha\|_2
+
t\|1_F\Delta_\mu^{(k)}P_t^{(k)}\alpha\|_2
\leqslant
C \e^{-c d(E,F)^2/t}\|1_E\alpha\|_2.
\end{align}

Let \(u=P_t^{(k)}\alpha\), and let \(\varphi\in C_0^\infty(M)\) satisfy
\(\varphi=1\) on \(F\). By the Weitzenböck formula and the lower bound
\[
\langle \mathscr{R}_h^{(k)}\eta,\eta\rangle \geqslant -V_k|\eta|^2,  \ \eta\in \OO^{(k)}
\]
we have
\[
\int_M \varphi^2|\nabla u|^2\,d\mu
\leqslant 
\int_M\varphi^2\langle \Delta_\mu^{(k)}u,u\rangle\,d\mu
+
\int_M V_k\varphi^2|u|^2\,d\mu
+
2\int_M\varphi |\nabla\varphi|\,|\nabla u|\,|u|\,d\mu .
\tag{3.10}
\]
Applying the Kato form bound \eqref{V-ineq} to the scalar function
\[
f=\varphi |u|
\]
and using Kato's inequality
\[
|\nabla |u||\leqslant|\nabla u|,
\]
we obtain, for some \(\kappa\in(0,1)\),
\[
\int_M V_k\varphi^2|u|^2\,d\mu
\leqslant
\kappa\int_M\varphi^2|\nabla u|^2\,d\mu
+
C\int_M|\nabla\varphi|^2|u|^2\,d\mu
+
C\int_M\varphi^2|u|^2\,d\mu.
\]
Using Young's inequality to absorb the terms containing
\(\varphi^2|\nabla u|^2\) into the left-hand side, we get
\[
\int_M\varphi^2|\nabla u|^2\,d\mu
\leqslant
C\int_M\varphi^2|\Delta_\mu^{(k)}u|\,|u|\,d\mu
+
C\int_M(\varphi^2+|\nabla\varphi|^2)|u|^2\,d\mu .
\]
Multiplying by \(t\), using Cauchy's inequality, and then applying the off-diagonal estimates
\eqref{L2-Gaffney} to \(u=P_t^{(k)}\alpha\) and \(\Delta_\mu^{(k)}u\), one obtains
\[
\|1_F\sqrt t\,|\nabla P_t^{(k)}\alpha|\|_2
\leqslant
C(1+\sqrt t)
\e^{-c d(E,F)^2/t}\|1_E\alpha\|_2 .
\]
This proves the desired estimate.
\end{proof}

\subsection{Proof of Theorem \ref{main-them1}}
To begin our discussion, we need the following lemma taken from  \cite[Section 4]{PTTS-2004}.
\begin{lemma}[\cite{PTTS-2004}]\label{lem-fop}
  If  \eqref{A1} holds, then there exist $N_0\in \mathbb N$ depending only on $A$ and $m$, and  a countable subset
  $ \{x_j\}_{j\geqslant 1}\subset M$, such that
  \begin{itemize}
  \item [{\rm (i)}] $M=\cup_{j\geqslant 1} B(x_j,1)$;

  \item [{\rm (ii)}] $\big\{B(x_j,1/2)\big\}_{j\geqslant 1}$ are disjoint;

  \item [{\rm (iii)}]  for every $x\in M$, there are at most $N_0$ balls $B(x_j,4)$ containing $x$;

  \item [{\rm (iv)}] for any $c_0\geqslant 1$, there exists a constant $C>0$ depending only on $c_0,A$ and $m$, such that for any $j\geqslant 1$ and $x\in B(x_j,c_0)$, 
    \begin{align*}
     & \mu\lf(B(x,2r)\cap B(x_j,c_0)\r)\leqslant C \mu\lf(B(x,r)\cap B(x_j,c_0)\r),\quad r\in (0,\infty), \\
      & \mu(B(x,r))\leqslant C \mu\lf(B(x,r)\cap B(x_j,c_0)\r),\quad r\in (0,2c_0].
    \end{align*}

  \end{itemize}
\end{lemma}

For $p\in (2,\infty)$ and $\si\in (A,\infty)$, we intend to find   $C\in (0,\infty)$ depending only on $p,\si, A,m$ and $V_k$ such that 
\begin{align}\label{eqn-Tpg2}
  \big\Vert \, |\T_{\mu,\sigma}^{(k)} (\alpha)|\, \big\Vert_p \leqslant C \Vert \alpha \Vert_p,\quad \aa\in \Omega^{(k)}.
\end{align} 
To this end, let $w$ be a $C^{\infty}$ function on $[0,\infty)$ satisfying $0\leqslant w\leqslant 1$ and
\begin{align*}
  \begin{split}
    w(t)= \left\{
      \begin{array}{ll}
        1 & \text{on} ~ ~ [0, 1/2], \\
        0 & \text{on} ~ ~ [1,\infty) \text{,}
      \end{array}
    \right.
  \end{split}
\end{align*}
and let $\widetilde{\T}_{\mu,\sigma}^{(k)}$ be the operator defined by
\begin{align}\label{eqn-wzT}
\widetilde{\T}_{\mu,\sigma}^{(k)}( \alpha) := \int_{0}^{\infty}v(t)\nabla  P_{t}^{(k)} \alpha \,d t
\end{align}
where $v(t):= {w(t) \e^{-\sigma t} }/{\sqrt{t}}$.
We need the following lemma, which reduces \eqref{eqn-Tpg2} to a time and spatial localized version.

\begin{lemma}\label{lem-tpg2}
  Suppose that  Condition  {\bf (A)} holds. Let $p\in(2,\infty)$ and $\{x_j\}_{j\geqslant 1}$ be as  in Lemma \ref{lem-fop}. If there exists a   constant $c>0$ depending only on $p,\si,A,m$ and $V_k$ such that
  \begin{align}\label{eqn-wTfb}
    \left\Vert\, |\widetilde{\T}_{\mu,\sigma}^{(k)}(\alpha)|\, \right\Vert _{L^p(B(x_j,4))}\leqslant c \Vert\alpha\Vert_{L^p(B(x_j,1))}
  \end{align}
  for any $\alpha\in L^p(\OO^{(k)}, \mu)$, then inequality \eqref{eqn-Tpg2} holds for some constant $C>0$ depending also only on $p,\si,A,m$ and~$V_k$.
\end{lemma}
\begin{proof} In the sequel, $\xi\lesssim \eta$ for two positive variables $\xi$ and $\eta$ means that $\xi \leqslant \kappa \eta$ holds for some constant $\kappa>0$ depending only on $p,\si,A,m$ and $V_k$.

 Since  $w \equiv 1$ on $[0,1/2]$,  if $\sigma >A$,  then  \eqref{A3} implies that for any $\alpha \in L^p(\OO^{(k)},\mu)$,
  $$
  \left\Vert \, \int_{0}^{\infty}(1-w(t))\, |\nabla P^{(k)}_{t}\alpha|\, \frac{\e^{-\sigma t}}{\sqrt{t}} \, \d t\right\Vert _p \lesssim \int_{1/2}^{\infty}\e^{(A-\sigma)t}\frac{1}{t}\d t \,\Vert\alpha\Vert _p \lesssim \Vert\alpha\Vert _p .
  $$
  (Note that the first inequality in condition \eqref{A3} extends to general $\alpha\in L^p(\OO^{(k)},\mu)$
  by a standard approximation argument in $L^p(\mu)$).
  This and \eqref{eqn-wzT} imply that  \eqref{eqn-Tpg2} follows from 
  \begin{align}\label{eqn-wTfb-1}
    \left\Vert |\wz {\T}^{(k)}_{\mu,\sigma}(\alpha)| \right\Vert _p \lesssim \Vert \alpha \Vert _p,\quad \alpha\in L^p(\OO^{(k)}, \mu).
  \end{align}

  Let $\{x_j\}_{j\geqslant 1}$ be as  in Lemma \ref{lem-fop}  and $\{\varphi_j\} $ be a subordinated $C^{\infty}$ partition of the unity such that $0\leqslant\varphi_j\leqslant 1$ and $\varphi_j$ is supported in $B_j := B(x_j,1)$. For each $j$, denote the characteristic function of the ball $4B_j:=B(x_j,4)$ by $\chi_j$.
  For any $\alpha\in L^p(\OO^{(k)}, \mu)$ and $x\in M$, we then may write
  \begin{align}\label{eqn-wztq}
    \widetilde{\T}_{\mu,\sigma}^{(k)} \alpha(x)\leqslant \sum_{j\geqslant 1}\chi_j\widetilde{\T}_{\mu,\sigma}^{(k)}(\alpha\varphi_j)(x)+\sum_{j\geqslant 1}(1-\chi_j)\widetilde{\T}_{\mu,\sigma}^{(k)}(\alpha\varphi_j)(x)=: \text{I}(x)+\text{II}(x).
  \end{align}
  By Lemma \ref{lem-fop}, we know
  $$
  \sum_{j\geqslant 1}\left|(1-\chi_j)(x)\varphi_j(y)\right|\leqslant N_0 \mathbbm{1}_{\{\rho(x,y)\geqslant 3\}}.
  $$
  First note by Lemma \ref{lem3}, along with the volume
doubling property \eqref{A1}, there exists $C\colon (0,\infty)\to(0,\infty)$ depending only on $A$ and $m$  such that
  \begin{align}\label{GaussianTerm}
  \int_{\{\rho(x,y)\geqslant 3\}} \frac{\e^{-{ \gamma \rho^2(x,y)}/{t}}}{\mu(y,\sqrt{t})}\, \mu( \d y) \leqslant C_\gg\e^{-\ff1{C_\gg t}},\quad t\in (0,1],\ \gg>0,\ x\in M.
  \end{align}
By this and H\"older's inequality, we find $h_1,c\colon (0,\infty)\to (0,\infty)$ depending only on $p$, $A$ and $m$ such that 
  \[\begin{aligned}
    \text{II}(x) & \leqslant \int_{0}^{1} v(t) \int _M
    \left| \nabla _x p_{t}^{(k)}(x,y) \right|
    \left( \sum _{j \in \Lambda } \left|(1- \chi _j )(x) \varphi _j(y)\right| \right)
     | \alpha(y)| \, \mu(\d y)\d t \\
     & \leqslant N_0 \int_{0}^{1}\frac{1}{\sqrt{t}}\int _{\{\rho(x,y)\geqslant3\}}\left|\nabla_x p_{t}^{(k)}(x,y)\right|\cdot |\alpha (y)|\, \mu (\d y) \d t \\
     & \leqslant  N_0 \int_{0}^{1}\frac{1}{\sqrt{t}}\int _{\{\rho(x,y)\geqslant3\}}\left|\nabla_x p_{t}^{(k)}(x,y)\right|\,\e^{\gamma \rho^2 (x,y)/pt}{\mu(y,\sqrt{t})}^{(p-1)/p}\, |\alpha (y)|\,
\frac{\e^{-\gamma \rho^2 (x,y)/pt}}{\mu(y,\sqrt{t})^{(p-1)/p}}
     \, \mu (\d y) \d t \\
     & \leqslant  h_1(\gg)  \int_{0}^{1}\left( \int _M \left|\sqrt{t}\nabla _x p_{t}^{(k)}(x,y)\right| ^p \e^{\gamma \rho^2 (x,y)/t} \left(\mu(y,\sqrt{t})\right) ^{p-1} |\alpha(y)| ^p \mu (\d y)\right) ^{1/p} \frac{\e^{-c_{\gg}/t}}{t}\,\d t.
  \end{aligned}\]
By Theorem \ref{Lp-estimate}, there exists $h_2\colon (0,1/8)\to (0,\infty)$ depending only on $p, A,m$ and $V_k$ such that
$$
  \int_M\left|\sqrt{t}\, |\nabla_x p_{t}^{(k)}(x,y)|\right| ^p \e^{  \frac{\gg\rho^2 (x,y)}{t}}\mu (\d x) \leqslant    \frac{h_2(\gg)}{\left(\mu(y,\sqrt{t})\right) ^{p-1}},\quad 0<\gg<1/8.
  $$
Taking for instance $\gg= 1/16$, we find constants $c_0,c_1,c_2,c_3 \in (0,\infty)$ depending only on $p,A,m$ and $V_k$ such that
  \begin{equation}\begin{split} \label{eqn-wztqp}
    &\int_M |\text{II}(x)|^p \mu (\d x) \\
    &\leqslant c_1  \int_M \left(\int_{0}^{1} \left(\int_M |\sqrt{t}\nabla _x p_{t}^{(k)}(x,y)|^p \e^{\gamma\rho^2(x,y)/t} \mu\big(y,\sqrt{t}\big)^{p-1}|\alpha(y)|^p \mu (\d y) \right) ^{1/p} \frac{\e^{-c_0/t}}{t} \d t \right) ^p \mu(\d x)  \\
     &\leqslant c_2 \int_{0}^{1} \left(\int_M \lf(\mu\big(y,\sqrt{t}\big)\r)^{p-1}|\alpha(y)|^p \left(\int_M \left|\sqrt{t}\nabla_x p_{t}^{(k)}(x,y)\right|^p \e^{\gamma \rho^2 (x,y)/t}\mu (\d x)\right) \mu (\d y)\right) \, \d t  \\
    &\leqslant c_3 \int_M |\alpha (y) |^p \mu (\d y). \end{split} 
  \end{equation}

  Next we turn to the estimate of $\text{I}(x)$. According to Lemma \ref{lem-fop}, the balls $\{4B_j\}_{j \in \Lambda}$ form a unity overlap and hence
$$
\sum_{j} \Vert\rho\chi _j \Vert _{p/(p-1)}^{p/(p-1)} \lesssim \Vert \rho\Vert_{p/(p-1)}^{p/(p-1)},
\quad \rho\in C_0^{\infty}(M).
$$
  Combined with assumption \eqref{eqn-wTfb}, since $|\alpha |\varphi _j \in C_0^{\infty}(B(x_j,1))$,
we conclude that
  \begin{align*}
    \left|\int_M \rho(x)|\text{I}(x)|\,\mu (\d x)\right| & \leqslant \int_M |\rho(x)| \, \bigg|\sum_j \chi_j \wz {\T}^{(k)}_{\mu,\sigma}(\alpha\varphi_j)(x)\bigg| \, \mu(\d x) \\
    & \lesssim \sum _j \big\Vert  \, |\alpha | \, \varphi _j \big\Vert _p \,\Vert \rho\chi_j \Vert _{p/(p-1)} \lesssim \Vert\alpha\Vert _p \Vert\rho\Vert_{p/(p-1)}.
  \end{align*}
This together with \eqref{eqn-wztq} and \eqref{eqn-wztqp} implies \eqref{eqn-Tpg2}, and concludes the proof.
\end{proof}

In the sequel, we continue to write
$B_j:=B(x_j,1)$ for simplicity.  
By Lemma \ref{lem-tpg2}, it suffices to verify \eqref{eqn-wTfb}. To this end, we use the local
$L^p$ boundedness criterion  via maximal functions from \cite{PTTS-2004}.
More precisely, we define the \textit{local maximal function}  by
\begin{align}\label{eqn-localMF}
  (\SM_{\loc}f)(x):=\sup_{\substack{x\in B \\
  r(B)\leqslant 32}}\frac{1}{\mu(B)}\int_B f\,\d\mu,\quad x\in M,
\end{align}
for any locally integrable function $f$ on $M$;
the supremum is taken over all balls $B$ in $M$, containing $x$ and
of radius at most 32. From \eqref{A1},
it follows that $\SM_{\loc}$ is bounded on $L^p(\mu)$ for all
$1<p\leqslant \infty$.  For a measurable subset $E\subset M$, the
\textit{maximal function relative to $E$\/} is defined as
\begin{align}\label{eqn-relaMF}
  (\SM_Ef )(x):=\sup_{B\text{ ball in }  M,\, x\in B} \frac{1}{\mu(B\cap E)}\int_{B\cap E}f\,\d\mu,
  \quad x\in E,
\end{align}
for any locally integrable function $f$ on $M$. If in particular $E$
is a ball of radius $r$, it is enough to consider balls $B$ with
radii not exceeding $2r$. It is also easy to see $\SM_E$ is weak
type $(1,1)$ and $L^p(\mu)$-bounded for $1<p\leqslant \infty$ if $E$
satisfies the {\it relative doubling property}, namely, if there
exists a constant $C_E$ (called \textit{relative doubling constant of} $E$)
such that for $x\in E$ and $r>0$,
\begin{align}\label{RD}
  \mu(B(x,2r)\cap E)\leqslant C_E\,\mu(B(x,r)\cap E).
\end{align}
Note that by Lemma \ref{lem-fop} (iv), for any $j\in \Lambda$, in particular 
the subsets $4B_j$ satisfy the relative doubling property \eqref{RD}
with a relative doubling constant independent of $j$.\smallskip

The following lemma will be crucial in the
proof of Theorem \ref{main-them1}. For any $x\in M$, let $\scr B(x)$ be  the class of geodesic balls in $M$ containing $x$.

\begin{lemma}\label{theorem-add-2}
  Let $p\in (2,\infty)$ and assume \eqref{A1}. Then  \eqref{eqn-wTfb} holds for some   constant $c>0$ depending only on $p,\si,A,m$ and $V_k$, provided  there exist  an integer $n$ and  a constant $C>0$ depending only on $p,\si,A,m$ and $V_k$ such that the following two items hold: 
  \begin{enumerate}\item[\rm(i)]  the operator
  \begin{align*}
      \SM^{\#}_{4B_j, \widetilde{\T}^{(k)}_{\mu, \sigma},\, n}\alpha(x):=\dsup_{B\in \scr B(x)}\lf(\frac{1}{\mu(B\cap 4B_j)}\dint_{B\cap 4B_j}
      \lf|\widetilde{\T}^{(k)}_{\mu, \sigma} (I-P^{(k)}_{r^2})^n\alpha(y)\r|^2\,\mu(\d y)\r)^{1/2} 
    \end{align*} for $x\in 4 B_j$ satisfies  
    $$\big\|\SM^{\#}_{4B_j, \widetilde{\T}^{(k)}_{\mu, \sigma}, n}\alpha\big\|_{L^p(4B_j,\mu)}\leqslant C \|\aa\|_{L^p(\OO^{(k)}(B_j),\mu)},\ \ j\geqslant 1.$$ 
  \item[\rm(ii)] for any $\ell \in \{1,2,\ldots,n\}, j\geqslant 1$, and any
    $\alpha \in L^p(\OO^{(k)}(B_j),\mu),$  there exists a sublinear operator~$S_j$
    bounded from $L^p(\OO^{(k)}(B_j), \mu)$ to $L^p(4B_j,\mu)$ with 
    $$\|S_j\|_{L^p(\OO^{(k)}(B_j),\, \mu)\to  L^p(4B_j,\,\mu)}^{\mstrut}\leqslant C,$$
      such that 
    \begin{equation}\label{eq2}\begin{split} 
      \sup_{B\in \scr B(x)} &\left(\frac{1}{\mu(B\cap 4B_j)}\int_{B\cap 4B_j} \big| \widetilde{\T}^{(k)}_{\mu, \sigma} (P^{(k)}_{ \ell r^2}\alpha)\big|^p\,\d\mu\right)^{1/p}\\
      &\quad\leqslant C\left(\SM_{4B_j}(| \widetilde{\T}^{(k)}_{\mu, \sigma} \alpha |^2)+\big(S_j( \alpha)\big) ^2 \right)^{1/2}(x),\ \ j\geqslant 1,\ x\in  4B_j.\end{split}
    \end{equation} 
  \end{enumerate}
   
\end{lemma}

\begin{proof}
 We use \cite[Theorem 2.4]{PTTS-2004}:
 First note that we may take $B_j$ and $4B_j$ for $E_1$ and $E_2$ there, 
 respectively, as
the sets $B_j$ and $4B_j$ possess the relative volume doubling property
\eqref{RD} with relative doubling constants independent of $j$
(see Lemma \ref{lem-fop}). As in  \cite{PTTS-2004} consider the operators $\{A_r\}_{r>0}$ given
by the relation
\begin{align*}
  I-A_r=\big(I-P^{(k)}_{r^2}\big)^n,\quad r>0,
\end{align*}
for some integer $n$ (to be chosen later).
Following the proof of \cite[Theorem 2.4]{PTTS-2004}, replacing $f\in L^p(B_j,\mu)$ by $\alpha \in L^p(\OO^{(k)}(B_j),\mu)$, we find a constant $C'>0$ depending only on $p, \si,A,m$ and $V_k$ such that 
\begin{align*}
\| \SM_{4B_j} (|\widetilde{\T}^{(k)}_{\mu, \sigma}\alpha|^2)^{1/2} \|_{L^p(4B_j)}\leqslant C'\lf( \Big \|\SM^{\sharp}_{4B_j, \, \widetilde{\T}^{(k)}_{\mu, \sigma}, \, n} \alpha \Big \|_{L^p(4B_j)} +\| S_j(\alpha) \|_{L^p(4B_j)}
+\|\, \alpha\,\|_{L^p(4B_j)} \r).
\end{align*}
Thus, assuming $L^p$-boundedness of both $\SM^{\#}_{4B_j,\,\widetilde{\T}^{(k)}_{\mu, \sigma},\, n}$ and $S_j$, we may conclude that
$\SM_{4B_j} (|\widetilde{\T}^{(k)}_{\mu, \sigma}\alpha|^2)^{1/2} $  is bounded in $L^p(4B_j,\mu)$ and thus $\widetilde{\T}^{(k)}_{\mu, \sigma}$ bounded from $L^p(\OO^{(k)}(B_j),\, \mu)$ to $L^p(\OO^{(k)}(4B_j),\, \mu)$.
\end{proof}

Hence it suffices to check (i) and (ii) of Lemma \ref{theorem-add-2}. We
establish two technical lemmas which verify (i) and (ii)
respectively. To this end, observe that \eqref{A1} implies: for
any $r_0>0$ there exists $C_{r_0}>0$ depending only on $A,m$ and $r_0$ such that  
\begin{align}\label{ExpGrowth}
  \mu(x,2r)\leqslant C_{r_0} \mu(x,r),\ \ r\in (0,r_0),\ x\in M.
\end{align}
An immediate consequence of \eqref{A1} is that for all $y\in M$,
$0<r<8$ and $s\geqslant 1$ satisfying $sr<32$,
\begin{align}\label{local-doubling}
  \mu(y,sr)\leqslant Cs^{m} \mu(y,r),
\end{align}
for some constants $C$ depending only on $A$.\smallskip

The following lemma is essential to the proof of part (i) of Lemma
\ref{theorem-add-2}.

\begin{lemma}\label{lem-term1}
  Assume condition {\bf (A)}.  Then
  there exists an integer $n$ depending only on $m$  and a constant $C>0$ depending on $\si, A,m$ and $V_k$, such that 
  \begin{align}\label{eq-Hess-1}
    \dsup_{B\in  \scr B(x)}\lf(\frac{1}{\mu(B\cap 4B_j)}\dint_{B\cap 4B_j}
    \lf|\wz{\T}^{(k)}_{\mu, \sigma}(I-P^{(k)}_{r^2})^n\alpha(y)\r|^2\,\mu(\d y)\r)^{1/2}
    \leqslant C \left(\SM_{\loc}(|\alpha|^2)(x)\right)^{1/2} 
  \end{align}  holds for any $x\in 4B_j,~j\geqslant 1$ and $\alpha \in L^2(\OO^{(k)}(4B_j), \mu)$
  where $\SM_{\loc}$ is defined by \eqref{eqn-localMF}.
\end{lemma}

\begin{proof} All constants appearing below depend only on $\si, A,m$ and $V_k$, and $\xi\ls \eta  $ for positive variables $\xi$ and $\eta$ means that $\xi\leqslant \kappa\eta$ holds for such a constant $\kappa>0$. 

  Viewing the left-hand side of \eqref{eq-Hess-1} as maximal
  function relative to $4B_j$, since the radius of $4B_j$ is $4$, it
  is sufficient to consider balls $B$ of radii not exceeding~$8$.    By Lemma \ref{lem-fop}, there exists a constant $c_0>0$ depending only on $A,m$ such that 
  \begin{align}\label{B-control}
    \mu(B) \leqslant c_0 \mu(B\cap 4B_j),\ \ B=B(x_0,r),\ x_0\in 4 B_j,\ r\in (0,8),\ j\geqslant 1. 
  \end{align}
    Hence,   
  \begin{align*} & \frac{1}{\mu(B\cap 4B_j)}\int_{B\cap 4B_j} \big|\wz{\T}^{(k)}_{\mu, \sigma}(I-P^{(k)}_{r^2})^n \alpha\big|^2\,\d\mu\leqslant    \frac{c_0}{\mu(B)}\int_{B}\big|\wz{\T}^{(k)}_{\mu, \sigma}(I-P^{(k)}_{r^2})^n\alpha \big|^2\,\d\mu,\\
   & \ j\geqslant 1,\ B=B(x_0,r),\ x_0\in 4 B_j,\  r\in (0,8). 
  \end{align*}
  Thus, we only need to show that
  \begin{align}\label{eqn-lem-term1-1}
    \dsup_{\stackrel{B=B(x_0,r)\in \scr B(x)}{r<8}} \frac{1}{\mu(B)}\dint_{B}
    \lf|\wz{\T}^{(k)}_{\mu, \sigma}(I-P^{(k)}_{r^2})^n\alpha(y)\r|^2\,\mu(\d y) 
    \ls   \SM_{\loc}(|\alpha|^2)(x),\ \ j\geqslant 1,\ x\in 4 B_j.
  \end{align}
For any  $r\in (0,8)$, we may choose $i_r\in \Z_+$
  satisfying
  \begin{align}\label{def-ir}
    2^{i_r}r\leqslant 8< 2^{i_r+1}r.
  \end{align}
  Denote by
  \begin{align}\label{C-B}
    &\mathcal{D}_i:=  (2^{i+1}B)\setminus (2^iB)\quad \text{if}\  i\geqslant 2,\quad\text{ and\/}\notag \\ &\mathcal{D}_1=4B.
  \end{align}
  Using the fact that $\supp \alpha \subset 4B_j\subset 2^i B$ when $i>i_r$, we
  find that
  \begin{align*}
    \alpha=\dsum_{i=1}^{i_r} \alpha \mathbbm{1}_{\mathcal{D}_i}=:\dsum_{i=1}^{i_r} \alpha_i
  \end{align*}
  which then implies
  \begin{align}\label{eqn-lem-term1-2}
    \lf\|\,\big| \wz{\T}^{(k)}_{\mu, \sigma}(I-P^{(k)}_{r^2})^n\alpha\big|\,\r\|_{L^2(B)}\leqslant \dsum_{i=1}^{i_r}\lf\|\,\big|\wz{\T}^{(k)}_{\mu, \sigma}(I-P^{(k)}_{r^2})^n\alpha_i\big|\,\r\|_{L^2(B)}.
  \end{align}
  For $i=1$ we use the $L^2$-boundedness of $\wz{\T}^{(k)}_{\mu, \sigma}\left(I-P^{(k)}_{r^2}\right)^n$ to
  obtain
  \begin{align}\label{eqn-lem-term1-3}
    \lf\|\,\big|\wz{\T}^{(k)}_{\mu, \sigma}(I-P^{(k)}_{r^2})^n\alpha_1\big|\,\r\|_{L^2(B)}\leqslant \|\alpha\|_{L^2(4B)}\leqslant \mu(4B)^{1/2}(\SM_{\loc}(|\alpha|^2)(x))^{1/2}
  \end{align}
  as desired.  For $i\geqslant 2$, we infer from \eqref{eqn-wzT} that
  \begin{align*}
    \wz{\T}^{(k)}_{\mu, \sigma}\left(I-P^{(k)}_{r^2}\right)^n \alpha_i
    &=\int_0^{\infty}v(t)\nabla \left(P^{(k)}_t(I-P^{(k)}_{r^2})^n \alpha_i \right)\,\d t\\
    &=\dint_0^\fz v(t)\dsum_{\ell=0}^n \binom n\ell (-1)^{\ell} \nabla P^{(k)}_{t+\ell r^2}\alpha_i\,\d t \\
    &=\dint_0^\fz \lf(\dsum_{\ell=0}^n \binom n \ell (-1)^{\ell} \mathbbm{1}_{\{t>\ell r^2\}} v(t-\ell r^2)\r)\nabla\, P^{(k)}_t\alpha_i\,\d t\\
    &=\int_0^{\infty}g_r(t)\,\nabla P^{(k)}_t\alpha_i \, \d t,
  \end{align*}
  where
  $$g_r(t):=\dsum_{\ell=0}^n \binom n \ell  (-1)^{\ell} \mathbbm{1}_{\{t>\ell r^2\}} v(t-\ell r^2).$$
  For $g_r$, according to the definition $v(t)={w(t)}\e^{-\sigma t}/{\sqrt{t}}$
  along with
  an elementary calculation (see the proof of \cite[Lemma 3.1]{PTTS-2004}), we observe that
  \begin{align*}
    \begin{cases}
      |g_r(t)|\ls \frac{1}{\sqrt{t-\ell r^2}}, \quad  &\mbox{for}\ \  0<\ell r^2<t \leqslant \ (1+\ell)r^2\leqslant(1+n)r^2,\\
      |g_r(t)| \ls r^{2n} t^{-n-\frac{1}{2}},\quad  &\mbox{for}\ \   (1+nr^2)\wedge (1+n)r^2<t \leqslant 1+nr^2,\\
      g_r(t)=0,\quad & \mbox{for}\ \ t > 1+nr^2.
    \end{cases}
  \end{align*}
Combined with
  Lemma \ref{lem-Gaffney}, this gives
  \begin{align*}
    \lf\|\,\big|\wz{\T}^{(k)}_{\mu, \sigma}(I-P^{(k)}_{r^2})^n\alpha_i\big|\,\r\|_{L^2(B)}\ls \left(\int_0^{\infty}|g_r(t)|\left(1+\sqrt{t}\right)\,\e^{-{c_2 4^ir^2}/{t}}\,\frac{ \d t}{\sqrt{t}}\right)\|\alpha_i\|_{L^2(\mathcal{D}_i)}
  \end{align*}
for some constant $c_2$ from  \eqref{Gaffney-ineq},  where by the fact that $0<r<8$, we have
  \begin{align*}
    &\int_0^{\infty}\left(1+\sqrt{t}\right)|g_r(t)|\,\e^{-{c_2 4^ir^2}/{t}}\,\frac{\d t}{\sqrt{t}}\leqslant C_n  \int_{0}^{1+n r^2}|g_r(t)|\, \e^{-{c_2 4^i r^2}/{t}}\frac{\d t}{\sqrt{t}}\leqslant C_n'4^{-in},
  \end{align*}
  for some constant $C_n'>0$.
Now, since $r(2^iB)\leqslant 8$ when
  $1\leqslant i\leqslant i_r$, an easy consequence of the local doubling
  \eqref{local-doubling} is that
  \begin{align*}
    \mu(2^{i+1}B)\leqslant C2^{(i+1)m}\mu(B),
  \end{align*}
  with a constant $C$  independent of $B$ and
  $i$. Therefore, as $\mathcal{D}_i\subset 2^{i+1}B$,
  \begin{align*}
    \|\alpha_i\|_{L^2(\mathcal{D}_i)}\leqslant \mu(2^{i+1}B)^{1/2}\left(\SM_{\loc}(|\alpha|^2)(x)\right)^{1/2}\leqslant C2^{im/2}\mu(B)^{1/2}\lf(\SM_{\loc}(|\alpha|^2)(x)\r)^{1/2}.
  \end{align*}
  Using the definition of $i_r$, $r\leqslant 8$, and then choosing
  $2n>m/2$, we finally obtain
  \begin{align*}
    &\left\|\,\big|\wz{\T}^{(k)}_{\mu, \sigma}(I-P^{(k)}_{r^2})^n\alpha \big |\,\right\|_{L^2(B)}\leqslant C'\left(\sum_{i=1}^{i_r}2^{i(m/2-2n)}\right)\mu(B)^{1/2}\lf(\SM_{\loc}(|\alpha|^2)(x)\r)^{1/2},
  \end{align*} for some constant $C'>0$, so that \eqref{eqn-lem-term1-1} holds. Then 
 the proof is finished. 
\end{proof}

The following lemma is used to prove part (ii) of Lemma
\ref{theorem-add-2}.

\begin{lemma}\label{lem-small-radii}
In the situation of Theorem \ref{main-them1},  let the integer $i_r$ be defined by
  \eqref{def-ir}, and let $n\in \mathbb{N}$ be as in Lemma
  \textup{\ref{lem-term1}}. Then there exist  constants $c,C>0$ depending only on $p,\si, A,m$ and the Kato data of $V_k$, such that for
    any $i\geqslant 1, \ell \in \{1,\ldots, n\}, r\in (0,8)$,  $B=B(x_0,r)\in \scr B (x)$, and for  $\alpha\in L^2(\OO^{(k)},\mu)$ supported in $\mathcal{D}_i$ as in \eqref{C-B}, 
  \begin{align}
    &\left(\frac{1}{\mu(B)}\int_B|\nabla P^{(k)}_{\ell r^2} \alpha|^p\,\d\mu\right)^{1/p}\leqslant   \frac{C\e^{C \ell r^2-c 4^i}}{r} \left(\frac{1}{\mu(2^{i+1}B)}\int_{\mathcal{D}_i}|\alpha|^2\,\d\mu\right)^{1 /2},\   \label{Hess-f-esti2}
  \end{align} 
 for  $\alpha\in L^2(\OO^{(k)},\mu)$ supported in $2^{i_r+2}B$, 
  \begin{equation} \label{Hess-f-esti1} 
   \left(\frac{1}{\mu(B)}\int_B\big| \nabla P^{(k)}_{\ell r^2} \big(\alpha\big) \big|^p\,\d\mu\right)^{1/p}\leqslant C\e^{C\ell r^2} \sum_{i=1}^{i_r+1}\frac{\e^{-c 4^i}}{\sqrt{\mu(2^{i+1}B)}} \left[\left(\int_{\mathcal{D}_i}|\nabla \alpha|^2\,\d\mu\right)^{1 /2}+\left(\int_{\mathcal{D}_i} |\alpha|^2\,\d\mu\right)^{1/ 2}\right].
  \end{equation} 
 \end{lemma}

\begin{proof} All constants appearing below depend only on $p, A,m$ and $V_k$. 
 We first observe from condition \eqref{A3} that
\begin{align}\label{Hess-ineq2}
      \left(\int_B \lf|\nabla P_{t}^{(k)} \alpha \r| ^p\,\d\mu\right)^{1/p}
    &\leqslant\frac{1}{\sqrt{t}}\e^{A+A t}\left(\int_B \Big(P_t |\alpha| ^2\Big)^{p/2}(x)\,\mu(\d x)\right)^{1/p}.
    \end{align}
    We substitute $t=\ell r^2$ in estimate~ \eqref{Hess-ineq2} for $\ell\in\{1,2,\ldots, n\}$. As $r\in (0,8)$,
    there exists a positive constant $\tilde C$ depending on $n$ and $A$ such that
    \begin{align*}
      \left(\int_B|\nabla P^{(k)}_{\ell r^2} \alpha|^p\,\d\mu\right)^{1/p}
      & \leqslant \frac{\tilde C}{r}\left(\int_B \Big(P_{\ell r^2 }|\alpha|^2\Big)^{p/2}(x)\,\mu(\d x)\right)^{1/p}.
    \end{align*}
    By the off-diagonal heat kernel upper bound of $p_t(x,y)$, see \eqref{Gaussion-upper-bound},
    we have
  \begin{align*}
    p_t(x,y)\leqslant \frac{C
    \e^{\tilde{\sigma}_2 t}}{\mu(y,\!\sqrt{t})}\exp{\left(-c_0\frac{\rho^2(x,y)}{t}\right)},\quad x,y\in M,
  \end{align*}
  for some constants $C, \tilde{\sigma}_2>0$ and $c_0 \in (0,1/4)$.
  As a consequence, since $0< r<8$, we obtain for $x\in B$, a positive constant $C>0$ such that
  \begin{align*}
   P_{\ell r^2}(|\alpha|^2)(x)
    &\leqslant C\int_{\mathcal{D}_i}\mu\lf(y,\sqrt{\ell}r\r)^{-1}\exp\left(-c_0\frac{\rho^2(x,y)}{\ell r^2}+\tilde{\sigma}_2 \ell r^2\right)|\alpha|^2(y)\,\mu(\d y)\\
    &\leqslant C\e^{\tilde{\sigma}_2 \ell r^2-c_04^i/\ell }\int_{\mathcal{D}_i}\mu\lf(y,\sqrt{\ell}r\r)^{-1}|\alpha|^2(y)\,\mu(\d y).
  \end{align*}
  Moreover, for $y\in \mathcal{D}_i$, we have
  $2^{i+1}B\subset B(y, 2^{i+2}r)$, and then by \eqref{A1}, for $\ell\in \{1,2,\ldots, n\}$,
  \begin{align*}
    \frac{1}{\mu\lf(y,\sqrt{\ell}r\r)}\leqslant \frac{2^{m(i+2)}\e^{C2^{i+2}}}{\mu(y,2^{i+2}r)}\leqslant \frac{2^{m(i+2)}\e^{C2^{i+2}}}{\mu(2^{i+1}B)}.
  \end{align*}
  It follows that
  \begin{align}\label{GI-0}
    P_{\ell r^2}(|\alpha|^2)(x)\leqslant C\e^{\tilde{\sigma} \ell r^2-{c_04^i}/{\ell}}\left(\frac{2^{m(i+2)}\e^{C2^{i+2}}}{\mu(2^{i+1}B)}\int_{\mathcal{D}_i}|\alpha|^2\,\d\mu\right)
  \end{align}
  for all $x\in B$, and there exists $\alpha_1<c_0/n$ such that for all $\ell\in \{1,2,\cdots, n\}$.
  \begin{align}\label{gradient-ineq-esti}
    \left(\frac{1}{\mu(B)}\int_{B}\Big(P_{\ell r^2}(|\alpha|^2)\Big)^{p/2}\,\d\mu\right)^{1/p}\leqslant C\e^{\tilde{\sigma} \ell r^2-\alpha_1 4^i}\left(\frac{1}{\mu(2^{i+1}B)}\int_{\mathcal{D}_i}| \alpha|^2\,\d\mu\right)^{1/2}.
  \end{align}
 Combining \eqref{Hess-ineq2} and \eqref{gradient-ineq-esti}, we 
    complete the proof of \eqref{Hess-f-esti2}.\smallskip

  We next observe that condition \eqref{A3} yields
  \begin{align}\label{Hess-ineq}
  &  \left(
\int_B|\nabla P_t^{(k)}\alpha|^p\,d\mu
\right)^{1/p}
\le
\e^{A+At}
\left(
\int_B(P_t|\nabla\alpha|^2)^{p/2}\,d\mu
\right)^{1/p}
+
\e^{A+At}
\left(
\int_B(P_t|\alpha|^2)^{p/2}\,d\mu
\right)^{1/p}.
  \end{align}
   If $\alpha$ is supported in $2^{i_r+2}B:=\cup_{i=1}^{i_r+1}\mathcal{D}_i$, then from \eqref{GI-0}, there exists $\alpha_1>0$ such that
    \begin{align*}
    P_{\ell r^2}(|\alpha|^2)(x)
 & \leqslant   C\sum_{i=1}^{i_r+1}\left(\frac{\e^{\tilde{\sigma} \ell r^2-\alpha_14^i}}{\mu(2^{i+1}B)}\int_{\mathcal{D}_i}|\alpha|^2\,\d\mu\right),
  \end{align*}
  which implies
  \begin{align*}
  & \left(\frac{1}{\mu(B)}\int_{B}\Big(P_{\ell r^2}(|\alpha|^2)\Big)^{p/2}\,\d\mu\right)^{1/p}\leqslant C\sum_{i=1}^{i_r+1}\e^{\tilde{\sigma} \ell r^2}\left(\frac{\e^{-2\alpha_1 4^i}}{\mu(2^{i+1}B)}\int_{\mathcal{D}_i}| \alpha|^2\,\d\mu\right)^{1/2}.
  \end{align*}
 By the same reason, we have
  \begin{align*}
  \left(\frac{1}{\mu(B)}\int_{B}\Big(P_{\ell r^2}(|\nabla \alpha|^2)\Big)^{p/2}\,\d\mu\right)^{1/p}\leqslant C\sum_{i=1}^{i_r+1}\e^{\tilde{\sigma} \ell r^2}\left(\frac{\e^{-2\alpha_1 4^i}}{\mu(2^{i+1}B)}\int_{\mathcal{D}_i}|\nabla \alpha|^2\,\d\mu\right)^{1/2}.
  \end{align*}
  Altogether, these estimates yield
  \begin{align*}
    \left(\frac{1}{\mu(B)}\int_B \lf|\nabla P_{\ell r^2}^{(k)} \alpha \r|^p\d\mu\right)^{1/p}\!\leqslant C' \sum_{i=1}^{i_r+1}\e^{(A+\tilde{\sigma}) \ell r^2}\left[\left(\frac{\e^{-2\alpha_1 4^i}}{\mu(2^{i+1}B)}\int_{\mathcal{D}_i}|\alpha|^2\,\d\mu\right)^{1/2}+\left(\frac{\e^{-2\alpha_1 4^i}}{\mu(2^{i+1}B)}\int_{\mathcal{D}_i}|\nabla \alpha|^2\,\d\mu\right)^{1/2}\right]
  \end{align*}
  which completes the proof of \eqref{Hess-f-esti1}.
\end{proof}

With the help of the Lemmas \ref{theorem-add-2}, \ref{lem-term1} and \ref{lem-small-radii}, we are
now in position to finish the proof of Theorem \ref{main-them1}.

\begin{proof}[Proof of Theorem \ref{main-them1}] For simplicity, denote by $C,c$  positive constants   depending only on $p,\si, A,m$ and $V_k$, which may vary from one term to another. 

  By Lemma \ref{theorem-add-2}, we only need to show that under the given assumptions, items (i) and (ii) of
  Lemma \ref{theorem-add-2} hold true. We first verify item (i) of Lemma
  \ref{theorem-add-2}. Observe from Lemma \ref{lem-term1}, there
  exists an integer $n$ and a constant $C>0$   such that for all $j\geqslant 1$, $\alpha\in L^2(\OO^{(k)}(B_j),\mu)$ and
  $x\in 4B_j$,
  \begin{align*}
    \sup_{ B \in \scr B(x)} \frac{1}{\mu(B\cap 4B_j)}\int_{B\cap 4B_j} \bigl|\wz{\T}^{(k)}_{\mu, \sigma}(I-P^{(k)}_{r^2})^n\alpha\bigr|^2\,\d\mu\leqslant C  \SM_{\loc}(|\alpha|^2)(x).
  \end{align*}
 Recall that $\SM_{\loc}$ is bounded on $L^p(\mu)$ for
  $1<p\leqslant \infty$; thus   $\SM^{\#}_{4B_j,\wz {\T}^{(k)}_{\mu,\sigma},n}$ is bounded from $L^p(\OO^{(k)}(B_j),\mu)$ to $L^p(4B_j,\mu)$ uniformly in $j$, i.e., assertion (i) is proved.

  Next, we prove (ii) of Lemma \ref{theorem-add-2}.  Assume that
  $\alpha \in \Omega_0^{(k)}(B_j)$ and let $h=\int_0^{\infty}v(t)P^{(k)}_t \alpha\,\d t$ with $v$
  as  in \eqref{eqn-wzT}. Since $\wz{\T}^{(k)}_{\mu, \sigma}(\alpha) =\nabla\, h$
  and inequality \eqref{B-control} holds for $B\cap 4B_j$, we have
  \begin{align*}
   & \left(\frac{1}{\mu(B\cap 4B_j)}\int_{B\cap 4B_j} \big|\wz {\T}^{(k)}_{\mu,\sigma} P^{(k)}_{\ell r^2}\alpha \big |^p\,\d\mu \right) ^{1/p} \\
    &\quad = \left(\frac{1}{\mu(B\cap 4B_j)}\int_{B\cap 4B_j}\big|\nabla P^{(k)}_{\ell r^2}h\big|^p\,\d\mu\right)^{1/p}\\
    &\quad \leqslant C \left(\frac{1}{\mu(B)}\int_B \big|\nabla P^{(k)}_{\ell r^2}h\big|^p\,\d\mu\right)^{1/p}.
  \end{align*}
Let  $\varphi_0$ be a $C^{\infty}$ function supported in $2^{i_r+2}B$ with $\varphi_0(x)=1$ on $2^{i_r+1}B$ and $|\nabla \varphi_0|\leqslant 1/8$ as $8\leqslant 2^{i_r+1}r\leqslant 16$.
  We write
  \begin{align*}
    \nabla P^{(k)}_{\ell r^2}h=\nabla P^{(k)}_{\ell r^2}g_0+\sum_{i=i_r+1}^{\infty}\nabla P^{(k)}_{\ell r^2}g_i,
  \end{align*}
  where $g_0=h\varphi_0$ and $g_i=h(1-\varphi_0)\1_{\mathcal{D}_i}$. Next,
  we distinguish the two cases $i=0$ and $i> i_r$ where $i_r$
  is defined in \eqref{def-ir}. For the case $i=0$, since $g_0\in \Omega^{(k)}$ is supported in $2^{i_{r}+1}B$,  by the
  inequality \eqref{Hess-f-esti1} in Lemma \ref{lem-small-radii} and the definition of $\varphi_0$, we
  have
  \begin{align}\label{part-1}
    &\left(\frac{1}{\mu(B)}\int_B\big|\nabla P^{(k)}_{\ell r^2}g_0\big|^p\,\d\mu\right)^{1/p}\notag\\
    &\leqslant C\sum_{i=1}^{i_r+1}\e^{-c 4^i}\left(\left(\frac{1}{\mu(2^{i+1}B)}\int_{\mathcal{D}_i}|\nabla g_0|^2\,\d\mu\right)^{1/2}+\left(\frac{1}{\mu(2^{i+1}B)}\int_{\mathcal{D}_i}| g_0|^2\,\d\mu\right)^{1/2}\right)\notag\\
    &\leqslant C\sum_{i=1}^{i_r+1}\e^{-c 4^i}\left(\left(\frac{1}{\mu(2^{i+1}B)}\int_{\mathcal{D}_i}|\nabla h|^2\,\d\mu\right)^{1/2}+\left(\frac{1}{\mu(2^{i+1}B)}\int_{\mathcal{D}_i}| h|^2\,\d\mu\right)^{1/2}\right)\notag\\
    &\leqslant  C\sum_{i=1}^{i_r+1}\e^{-c 4^i}\left(\left(\SM_{\loc}(|\nabla h|^2)(x)\right)^{1/2}+\left(\SM_{\loc}(|h|^2)(x)\right)^{1/2}\right).
  \end{align}

  For the second regime $i> i_r$, we proceed with inequality \eqref{Hess-f-esti2} in Lemma \ref{lem-small-radii}   such that
  \begin{align}\label{Hess-gi}
    \left(\frac{1}{\mu(B)}\int_B|\nabla P^{(k)}_{\ell r^2}g_i|^p\,\d\mu\right)^{1/p}\leqslant \frac{C\e^{-c 4^i}}{r}\left(\frac{1}{\mu(2^{i+1}B)}\int_{\mathcal{D}_i}|h|^2\,\d\mu\right)^{1/2}.
  \end{align}
  On the other hand, since $i >i_r$, it is easy to see that
  $4B_j\subset2^{i+1}B$, thus
  \begin{align}\label{h-esti}
    \lf(\frac{1}{\mu(2^{i+1}B)}\int_{\mathcal{D}_i}|h|^2\,\d\mu\r)^{1/2}&\leqslant \left(\frac{1}{\mu(2^{i+1}B)}\int_0^1 v(t)\int_{\mathcal{D}_i} |P^{(k)}_t \alpha|^2\, \d\mu \, \d t \right)^{1/2}\notag\\
                                                                &\leqslant C\left(\frac{1}{\mu(4B_j)}\int_{B_j}|\alpha|^2\,\d\mu\right)^{1/2} \notag\\
                                                                &\leqslant C\Big(\SM_{4B_j}(|\alpha|^2)(x)\Big)^{1/2}.
  \end{align}
  Thus the contribution of the terms in the second regime $i> i_r$ is
  bounded by combining \eqref{Hess-gi} and \eqref{h-esti},
  \begin{align}\label{part-2}
    \sum_{i> i_r}\left(\frac{1}{\mu(B)}\int_B|\nabla P^{(k)}_{\ell r^2}g_i|^p\,\d\mu\right)^{1/p}\leqslant \sum_{i> i_r} \frac{C \e^{-c4^i}}{r}\big(\SM_{4B_j}(|\alpha|^2)(x)\big)^{1/2}
  \end{align}
  and it remains to recall that $1/r\leqslant 2^{i+1}/8$ when $i> i_r$.

  We conclude from \eqref{part-1} and \eqref{part-2} that for any
  $p>2$ and $\ell \in \{1,2,\ldots,n\}$, there exists a constant $C$
  independent of $j$ such that
  \begin{align*}
    \left(\frac{1}{\mu(B\cap 4B_j)}\int_{B\cap 4B_j}|\wz{\T}^{(k)}_{\mu, \sigma} P^{(k)}_{\ell r^2}\alpha|^p\,\d\mu\right)^{1/p}\leqslant C\Big(\SM_{4B_j}(|\wz{\T}^{(k)}_{\mu, \sigma}\alpha|^2)+(S_j\alpha )^2\Big)^{1/2}(x)
  \end{align*}
 for  all $\alpha \in L^2(\OO^{(k)}(B_j),\mu)$, all balls $B$ in $M$ and
  all $x\in B\cap 4B_j$, where the radius $r$ of $B$ is less than 8,
  and where
  \begin{align}\label{eqn-squref}
    (S_j\alpha)^2:=\SM_{\loc}\bigl(|\wz{\T}^{(k)}_{\mu, \sigma}\alpha|^2\mathbbm{1}_{M\setminus 4B_j}\bigr)
    +\SM_{\loc}\bigl(|h|^2\bigr)(x) +\SM_{4B_j}\bigl(|\alpha|^2\bigr).
  \end{align}

  Our last step is to show that the operator $S_j$ defined in
  \eqref{eqn-squref} is bounded from $L^p(\OO^{(k)}(B_j),\, \mu)$ to $L^p(4B_j,\, \mu)$ for any
  $p\in (2,\fz)$ with operator norm independent of $j$.  By
  \eqref{eqn-squref}, we only need to show that the operators
$$\big(\SM_{\loc}(|\wz{\T}^{(k)}_{\mu, \sigma}\alpha|^2\mathbbm{1}_{M\setminus 4B_j})\big)^{1/2},\ \
\big(\SM_{\loc}(|h|^2)\big)^{1/2} \ \mbox{ and }\
\big(\SM_{4B_j}(|\alpha|^2)\big)^{1/2}$$ respectively are bounded from
$L^p(B_j)$ to $L^p(4B_j)$.  Indeed, for any $\alpha \in L^p(4B_j)$, by Lemma
\ref{lem-fop} we know that $4B_j$ satisfies the doubling property
\eqref{A1}, which for $p>2$ implies that
$\big(\SM_{4B_j}(|\alpha|^2)\big)^{1/2}$ is bounded from $L^p(B_j)$ to
$L^p(4B_j)$ by a constant depending only on the doubling property
\eqref{A1}.  On the other hand, using the local estimate of $p_t^{(k)}(x, y)$ (\cite{CZ07}), we see that
$$|p_t^{(k)}(x, y)|\leqslant \frac{C}{\mu(x,\ss t)} \e^{-\frac{\gamma\rho(x,y)^2}{t}},\quad t\in (0,1],\ \gamma<1/4, $$
which together with \eqref{integ-esti2} and Cauchy's inequality implies
\begin{align*}
  \left\Vert P^{(k)}_t\alpha\right\Vert_p \leqslant C \|\alpha\|_p,\quad t\in (0,1].
\end{align*}
This,  together with  \eqref{A1} and   the $L^{p/2}$-boundedness of   $\SM_{\loc}(\cdot)$,
further implies
\begin{align*}
\big\|\lf(\SM_{\loc}(|h|^2)\r)^{1/2}\big\|_p
\leqslant C  \Big\|\int_0^{1}v(t)P^{(k)}_t \alpha\,\d t\Big\|_{p}\leqslant C\left(\int_0^{1}\frac{w(t)\,\e^{ -\sigma t }}{\sqrt{t}}\,\d t\right)\,\| \alpha\|_{L^p(B_j)}\leqslant C\| \alpha \|_{L^p(B_j)},
\end{align*}
for $p>2$ and $\sigma>0$.
Finally, the $L^p$-boundedness of
$$\big(\SM_{\loc}(|\wz{\T}^{(k)}_{\mu, \sigma}\alpha|^2\mathbbm{1}_{M\setminus
  4B_j})\big)^{1/2}$$ 
  follows from the $L^{p/2}$-boundedness of
$\SM_{\loc}(\newdot)$ and an argument similar to the $L^p$ boundedness of $\rm{II}$ in
\eqref{eqn-wztqp} since $\alpha\in \OO^{(k)}(B_j)$ and
 $$\mathbbm{1}_{M\setminus
  4B_j} \wz{\T}^{(k)}_{\mu, \sigma}\alpha = (1-\chi_j) \wz{\T}^{(k)}_{\mu, \sigma}(\alpha \varphi_j).$$ 
  This implies that the operator $S_j$ is bounded
from $L^p(B_j)$ to $L^p(4B_j)$ with an upper bound independent of~$j$.

We infer that the requirements (i) and (ii) in Lemma
\ref{theorem-add-2} both hold true under the assumptions \eqref{A1}, \eqref{A2} \, and \,  \eqref{A3}.
Thus, the operator $\wz{\T}^{(k)}_{\mu, \sigma}$ is bounded from $L^p(\OO^{(k)}(B_j),\mu)$ to
$L^p(\OO^{(k)}(4B_j),\mu)$ for $p>2$ with a constant independent of~$j$.
Therefore, by Lemma \ref{lem-tpg2}, the operator $\T^{(k)}_{\mu,\sigma}$ is strong type
$(p,p)$ for $p>2$. This concludes the proof of Theorem~\ref{main-them1}.
\end{proof}

\section{$L^p$-boundedness under curvature conditions}
\subsection{Proof of Theorem \ref{C1}}

By Theorem \ref{main-them1},
it suffices to  verify conditions \eqref{A1}, \eqref{A2} \, and \,  \eqref{A3} by using {\bf (C)}. By the Laplacian comparison theorem presented in \cite{Qian} and Lemmas 2.1-2.2 in \cite{GW01},
\eqref{A1} follows from the curvature-dimension condition \eqref{CD}.
Moreover, according to \cite{GW01}, \eqref{A2} is a consequence
of \eqref{CD} as well.
Thus, it remains to prove \eqref{A3}, which is Proposition \ref{cor} below.

\subsection{Derivative formulas}

Let $X_t(x)$ be diffusion process on $M$ generated by $L:=-\DD+\nn h$ with a fixed initial value $x\in M$, and let $u_t(x)$ be the horizontal lift of $X_t(x)$ to $O(M)$, such that
$$\d X_t(x)= \nn h(X_t(x))\,\d t+ \ss 2\, u_t(x)\circ \d B_t,\quad t\geqslant 0,\ X_0(x)=x,$$
where $B_t$ is an $m$-dimensional Brownian  motion on $\R^m$.
Then the associated stochastic parallel displacement is defined as
$$\ptr_{t,x}:= u_t(x) \, u_0(x)^{-1}\colon T_xM\to T_{X_t(x)}M,$$
where as usual orthonormal frames $u$ at a point $x$ are read as isometries
$u\colon \R^m\to T_xM$.
For fixed $k\in\N$, let $E:=\Lambda^{k}T^*M$ and $\tilde E:=T^*M\otimes\Lambda^{k}T^*M$. We are now in position to introduce the derivative formula for $P_t^{(k)}$.
To this end, let 
$$
\tilde{\mathscr{R}}_{h}^{(k)}=(\Ric-\Hess\, h)^{\rm tr}\otimes 1_{E}-2R^{(k)}\smallbullet + 1_{T^*M} \otimes \mathscr{R}^{(k)}_h  \in \End(\tilde E) ,
$$
where
$(\Ric -\Hess h)^{\rm tr}$ is the transpose of the Bakry-\'Emery Ricci curvature tensor $\Ric-\Hess\,h \in \Gamma (\End TM).$
 Let $Q_t\in \End( E_x)$ and $\tilde{Q}_t\in \End (\tilde E_x)$ denote the solutions to the ordinary differential equations
\begin{align*}
&\frac{\d}{\d t} Q_t=-Q_t  (\mathscr{R}^{(k)}_h)_{\ptr_{t,x}}, \ \ \ \  \  \ t\geqslant 0,\   \  Q_0=\text{id} _{E_x},\\
&\frac{\d}{\d t} \tilde{Q}_t=-\tilde{Q}_t (\tilde{\mathscr{R}}_h^{(k)})_{\ptr_{t,x}}, \ \ \ t\geqslant 0,\  \  \tilde{Q}_0=\text{id} _{\tilde E_x},
\end{align*}
where
\begin{align*}
& (\mathscr{R}^{(k)}_h)_{\ptr_{t,x}}= \ptr_{t,x}^{-1}\circ  \mathscr{R}^{(k)}_h \circ \ptr_{t,x},\quad  \mbox{and} \quad \  (\tilde{\mathscr{R}}_h^{(k)})_{\ptr_{t,x}} =\ptr_{t,x}^{-1}\circ  \tilde{\mathscr{R}} _h^{(k)} \circ \ptr_{t,x}.
\end{align*}
Let  $\mathscr{Q}_{\bolddot}$ and $\tilde{\mathscr{Q}}_{\bolddot}$  be the transposes of ${Q}_{\bolddot}$ and $\tilde{{Q}}_{\bolddot}$  respectively.

Moreover,  we have the commutation relation (see \cite[Proposition 2.15]{DT01})
$$\nn \DD_\mu^{(k)} = \tilde{\DD}_\mu^{(k)}\nabla + H^{(k)},$$
where $\tilde{\DD}_\mu^{(k)}:=\tilde\Box_{\mu}+\tilde{\mathscr{R}}^{(k)}_h$
with $\tilde\Box_{\mu}$ the Bochner Laplacian on $T^*M\otimes
E$ with respect to the induced connection on $T^*M\otimes
E$ and
$$
H^{(k)}:= \nabla \cdot R^{(k)}+R^{(k)}(\nabla h)+ \nabla \mathscr{R}_h^{(k)}\in \Gamma (\Hom(E, \, T^*M\otimes E).
$$
Let $H^{(k), {\rm tr}}$ be the transpose of the tensor $H^{(k)}$.
Finally let
\beg{align*} \tt P_t^{(k)} :=\e^{-t\tt\DD_\mu^{(k)} }, \quad t\geqslant 0.\end{align*}
For $\eta\in\Omega^{(k)}$, we define $\nn \eta\in\Gamma( T^*M\otimes E)$ by letting
$$\nn \eta (v)  :=  \nn_{v}\eta ,\quad v\in TM.$$

Assume that there exists a non-negative Kato function \(U\in\mathcal K\) such that
\begin{align}\label{Kato-U}
\langle \mathscr{R}_h^{(k)}\eta,\eta\rangle \geqslant -U|\eta|^2,
\qquad
\langle \tilde{\mathscr{R}}_h^{(k)}\zeta,\zeta\rangle \geqslant -U|\zeta|^2,
\qquad
|H^{(k)}|\leqslant U.
\end{align}
Then
\[
|Q_t|
\leqslant
\exp\left(\int_0^tU(X_s)\, \d s\right),
\qquad
|\tilde Q_t|
\leqslant
\exp\left(\int_0^tU(X_s)\, \d s\right).
\]
Therefore one needs the following Kato--Khasminskii estimate:
for every \(q\geqslant 1\), there exist constants \(C_q,A_q>0\) such that
\begin{align}\label{KKestimate1}
\sup_{x\in M}
\mathbb E^x
\exp\left(
q\int_0^tU(X_s)\, \d s
\right)
\leqslant 
C_qe^{A_qt},
\qquad t>0.
\end{align}
Moreover,
\begin{align}\label{KKestimate2}
\sup_{x\in M}
\mathbb E^x
\left[
\left(\int_0^tU(X_s)\, \d s\right)^q
\exp\left(
q\int_0^tU(X_s)\, \d s
\right)
\right]
\leqslant
C_q \e^{A_qt}.
\end{align}
These estimates are the main additional ingredient needed to make the Kato-class argument rigorous.

We have the following result.

\begin{proposition}\label{cor} Assume condition {\bf (C)} holds for some $k\in \N^+$. Then for any bounded $\eta\in \OO^{(k)}_{b,1}$,  there exists a constant $A>0$ such that for any $t>0$,
\begin{align}\label{FundEstimate}
  |\nn P_t^{(k)}\eta |\leqslant  \e^{A+A t} \min\bigg\{ t^{-1/2} \,  \left(P_t|\eta |^2\right)^{1/2},\ (P_t|\nabla \eta |^2)^{1/2}  +  (P_t |\eta |^2)^{1/2}  \bigg\}.
\end{align}
   \end{proposition}

 \beg{proof}
Consider for   $s\in [0,t]$:
\beg{align*}&N_s:=Q_s{/\!/ }_{s,x}^{-1} P_{t-s}^{(k)} \eta (X_s(x)),\\
& \tilde N_s:=\tilde{Q}_s{/\! / }_{s,x}^{-1} \nabla P_{t-s}^{(k)} \eta \left(X_s(x)\right).\end{align*} The crucial observation \cite[Theorem 3.7]{DT01} is that
\begin{align}\label{MartCrucial}
  Z_s^{(k)}:=\langle\tilde N_s,\xi_s\rangle-\langle N_s, U_s^{(k)}\rangle
\end{align}
is a local martingale where
$$U_s^{(k)}:=\int_{0}^s  \mathscr{Q}_r^{-1} \tilde{\mathscr{Q}}_r\dot\xi_r \,\d B_r + \int_{0}^s \mathscr{Q}_r^{-1} H_{\ptr_{r,x}}^{(k),\tr}\tilde{\mathscr{Q}}_r\xi_r \,\d s$$
and where $\xi_s$ may be any adapted process with absolutely continuous paths,
taking values in $T^*_xM\otimes E_x^{(k)}$. For simplicity, in the sequel, we always take
$\xi_s=\ell_s\xi$ for some fixed vector $\xi\in T^*_xM\otimes E_x^{(k)}$ and $\ell_s$
real-valued with absolutely continuous paths.
This leads to the local martingale
\begin{align}\label{MartCrucialnew}
  Z_s^{(k)}&:=\ell_s\langle\tt{Q}_s {/\!/}_{s,\, x}^{-1} \nabla P_{t-s}^{(k)}\eta (X_s(x)),\xi\rangle\notag\\
  &\quad-\left\langle  \ptr_{s, \, x}^{-1}P_{t-s}^{(k)}\eta(X_s(x)), \mathscr{Q}_s\int_{0}^s \dot\ell_r \mathscr{Q}_r^{-1} \tilde{\mathscr{Q}}_r\xi  \, \d B_r + \mathscr{Q}_s\int_{0}^s \ell_r\mathscr{Q}_r^{-1} H _{\ptr_{r,x}}^{(k),\tr}\tilde{\mathscr{Q}}_r\xi \,\d r\right\rangle.
\end{align}

When exploiting the martingale property of \eqref{MartCrucialnew}, there are different strategies for the choice of $\ell_s$ leading to different types of stochastic formulas for the covariant derivative $\nabla P_t^{(k)}\eta$.\smallskip

(a)\ (\textit{First upper bound in \eqref{FundEstimate}}) \
If $\ell $ is a bounded adapted process with paths in the Cameron-Martin
space $L^{2}([0,t];[0,1])$  such that $\ell(0)=1$ and $\ell(r)=0$ for $r\geqslant \tau \wedge t$, where $\tau=\tau_D(x)$ is the first exit time of $X_s(x)$ from some
relatively compact neighborhood $D$ of $x$, then trivially the local martingale \eqref{MartCrucialnew} is a true martingale and by taking expectations (see \cite[Section 4]{DT01}) 
the local covariant Bismut formula holds,
\begin{align}
&\left\langle \nabla P_t^{(k)}\eta, \xi  \right\rangle(x) \label{localBismut}\\
&=-\mathbb{E} \left[ \left\langle  \ptr_{t\wedge \tau, \, x}^{-1}P_{t-{t\wedge \tau}}^{(k)}\eta (X_{t\wedge \tau}(x)), \mathscr{Q}_{t\wedge \tau}\int_{0}^{t\wedge \tau} \dot\ell_s \mathscr{Q}_s^{-1} \tilde{\mathscr{Q}}_s\xi  \, \d B_s + \mathscr{Q}_{t\wedge \tau}\int_{0}^{t\wedge \tau} \ell_s\mathscr{Q}_s^{-1} H _{\ptr_{s,x}}^{(k),\tr}\tilde{\mathscr{Q}}_s\xi \,\d s\right\rangle\right].\notag
\end{align}
Under the condition {\bf (C)},  let $U:=3K+K_0$, then  \eqref{Kato-U} holds. 
By domination of Schr\"odinger semigroups,
\[
|P_s^{(k)}\eta|
\leqslant 
P_s^U|\eta|,
\]
where \(P_s^U\) denotes the scalar Schr\"odinger semigroup generated by
\(-\Delta_\mu+U\). By the Feynman--Kac formula,
\[
P_s^U|\eta|(y)
=
\mathbb E^y
\left[
\exp\left(
\int_0^sU(X_r)\, \d r
\right)
|\eta|(X_s)
\right].
\]
Then one derives the estimate
\begin{align*}
&-\mathbb{E} \lf[ \left\langle  \ptr_{t\wedge \tau, \, x}^{-1}P_{t-{t\wedge \tau}}^{(k)}\eta (X_{t\wedge \tau}(x)), \mathscr{Q}_{t\wedge \tau}\int_{0}^{t\wedge \tau} \dot\ell_s \mathscr{Q}_s^{-1} \tilde{\mathscr{Q}}_s\xi  \, \d B_s  \right \rangle  \right] \\
&\leqslant \mathbb{E} \lf[\big|P_{t-{t\wedge \tau}}^{(k)}\eta (X_{t\wedge \tau}(x))\big| \Big|\mathscr{Q}_{t\wedge \tau}\int_{0}^{t\wedge \tau} \dot\ell_s \mathscr{Q}_s^{-1} \tilde{\mathscr{Q}}_s\xi  \, \d B_s \Big|\r] \\
&\leqslant  \E \lf[P^{U}_{t-{t\wedge \tau}}|\eta| (X_{t\wedge \tau}(x)) \e^{\int_0^{t\wedge \tau}U(X_s)\, \d s}\, \Big|\int_{0}^{t\wedge \tau} \dot\ell_s \mathscr{Q}_s^{-1} \tilde{\mathscr{Q}}_s\xi  \, \d B_s \Big| \r] \\
&\leqslant (P_t |\eta|^2)^{1/2} \E \lf[ \e^{2\int_0^t U(X_s)\, \d s} \, \Big|\int_{0}^{t\wedge \tau} \dot\ell_s \mathscr{Q}_s^{-1} \tilde{\mathscr{Q}}_s\xi  \, \d B_s \Big|^2\r]^{1/2}.
\end{align*}
Let
\[
M_t
:=
\int_0^{t\wedge\tau}
\dot\ell_s\mathscr{Q}_s^{-1}\tilde{ \mathscr{Q}}_s\xi\,dB_s.
\]
Then, by Cauchy--Schwarz inequality, one gets
\[
\mathbb E^x
\left[
\e^{2\int_0^{t}U(X_r)\, \d r}|M_t|^2
\right]
\leqslant \mathbb E^x
\left[
\e^{4\int_0^{t}U(X_r)\, \d r}
\right]^{1/2}
\mathbb E^x
\left(
|M_t|^4
\right)^{1/2}.
\]
The first factor is bounded by \(C\e^{At}\) by \eqref{KKestimate1}. By the Burkholder--Davis--Gundy inequality,
\begin{align*}
\mathbb E^x|M_t|^4
&\leqslant 
C
\mathbb E^x
\left(
\int_0^{t\wedge\tau}
|\dot\ell_s|^2
|\mathscr{Q}_s^{-1}\tilde{ \mathscr{Q}}_s\xi|^2\, \d s
\right)^2\\
& \leqslant C\mathbb E^x
\left( 
\int_0^{t\wedge\tau}
|\dot\ell_s|^2 \e^{2\int_0^s U(X_r)\, \d r} \, \d s
\right)^2.
\end{align*}
Then by Cauchy--Schwarz inequality, we further estimate
\begin{align*}
\mathbb E^x|M_t|^4
& \leqslant Ct\, \mathbb E^x \lf(\int_0^{t} |\dot\ell_{s\wedge\tau}|^4    \e^{4\int_0^{s\wedge\tau} U(X_r)\, \d r} \, \d s\r)\\
& \leqslant Ct\, 
\lf(\mathbb E^x \int_0^{t}
|\dot\ell_{s\wedge\tau}|^8 \, \d s\r)^{1/2}  \lf(\int_0^t \E^x \e^{8\int_0^{s\wedge\tau} U(X_r)\, \d r} \, \d s\r)^{1/2}.
\end{align*} 
Using  the Kato--Khasminskii estimate, we get
\[
\mathbb E^x
\left[
\e^{2\int_0^tU(X_r)\, \d r}|M_t|^2
\right]
\le
C\e^{At} t^{\frac{3}{4}}
\left(
\mathbb E^x
\int_0^{t\wedge\tau}
|\dot\ell_s|^8\, \d s
\right)^{1/4}.
\]
On the other hand, we  use the Kato--Khasminskii estimate to obtain
\begin{align}\label{esti-KK}
\left(
\mathbb E^x
\left|
\e^{\int_{0}^t U(X_s)\, \d s} \int_{0}^{t\wedge \tau} \ell_s\mathscr{Q}_s^{-1} H _{\ptr_{s,x}}^{(k),\tr}\tilde{\mathscr{Q}}_s\xi \,\d s
\right|^2
\right)^{1/2}
\le
C\e^{At}.
\end{align}
This gives
\[
|\nabla P_t^{(k)}\eta|(x)
\le
C\e^{At}
(P_t|\eta|^2)^{1/2}
\left[
\left( t^{3}
\mathbb E^x
\int_0^{t\wedge\tau}|\dot\ell_s|^8\, \d s
\right)^{1/8}
+1
\right].
\]
To make this estimate more explicit, we choose a geodesic ball $D$ of radius $\delta_x$ centered at $x$.
It has been shown in \cite{Thalmaier-Wang:98} that there exists $\ell\in L^{2}([0,t];[0,1])$  such that $\ell(0)=1$ and $\ell(r)=0$ for $r\geqslant \tau \wedge t$ such that
\begin{align*}
\E\lf(\int_0^{t\wedge \tau}|\dot \ell_s |^8\, \d s\r) \leqslant  \frac{1}{t^7}\e^{8c(f)t} ,
\end{align*}
where $c(f):=\sup_D\left\{-f\Delta_{\mu}f+9|\nabla f|^2\right\}<\infty$ and   $f\in C^2(D)$ such that $f(x)=1$ and $f|_{\partial D}=0$. Specifically we may take
\begin{align*}
f(p)=\cos \lf(\frac{\pi \rho(x,\, p)}{2\delta_x}\r).
\end{align*}
Then using the comparison theorem in \cite[Theorem 1]{GW01}, it is easy to see that there exist positive constants $c_1(K,N)$ and $c_2(N)$ such that
\begin{align*}
  c(f)\leqslant \frac{c_1(K,N )}{\delta_x}+\frac{c_2(N)}{\delta_x^2}.
 \end{align*}
 Letting $\delta_x$ tend to $\infty$, we prove that
 \begin{align}
  \big|\nn P_t^{(k)}\eta \big|\leqslant C \e^{A t} t^{-1/2}  (P_t|\eta |^2)^{1/2}. \label{blowupat0}\end{align}

 (b)\ (\textit{Second upper bound in \eqref{FundEstimate}}) \
 We first prove the remaining claim of Proposition \ref{cor} for compactly
 supported $\eta$, i.e., for $\eta \in \OO_0^{(k)}$. To this end, we establish
 an estimate for $|\nabla P_{t}^{(k)}\eta|$ which is uniform in the time variable for small values of $t$.
For $\eta \in \OO_0^{(k)}$, the Kolmogorov equation gives
\begin{align*}
P_{t}^{(k)}\eta = \eta - \int_0^t P_s^{(k)} \Delta_{\mu}^{(k)} \eta \, \mathrm{d} s,
\end{align*}
which by \eqref{blowupat0} implies
\begin{align}
|\nabla P_{t}^{(k)}\eta| &\leqslant |\nabla \eta| + \int_0^t |\nabla P_s^{(k)} \Delta_{\mu}^{(k)} \eta| \,\mathrm{d} s \notag\\
&\leqslant |\nabla \eta| + c \int_0^t \e^{A s} s^{-1/2} \left( P_s | \Delta_{\mu}^{(k)} \eta |^2 \right)^{1/2} \, \mathrm{d} s \notag\\
&\lesssim \|\nabla \eta\|_{\infty} + c \e^{A t} \|\Delta_{\mu}^{(k)} \eta\|_{\infty} \sqrt{t}.\label{EstUniform}
\end{align}
Hence, $\sup_{s \in [0,t]} |\nabla P_{s}^{(k)}\eta| < \infty$. Also note that
there exists $A > 0$ such that
\begin{align*}
\sup_{s \in [0,t]} \left| \tt{Q}_s {/\!/}_{s,\, x}^{-1} \nabla P_{t-s}^{(k)}\eta (X_s(x)) \right| \leqslant \e^{A + A t} \left( \|\nabla\eta\|_{\infty} + \|\Delta_{\mu}^{(k)} \eta\|_{\infty} \right) < \infty
\end{align*}
for all $\eta \in \OO_0^{(k)}$.
As a consequence of these bounds, we conclude that the local martingale \eqref{MartCrucialnew} is a true martingale for the constant function $\ell_s\equiv1$ as well. Taking expectations at the endpoints $0$ and $t$,
we derive the following global Bismut formula, i.e.,
\begin{align}\label{BE-2}
\left \langle \nabla P_t^{(k)}\eta, \xi \right \rangle (x) = - \E \left\langle\ptr_{t,x}^{-1} \nabla \eta\left(X_t(x)\right), \tilde{\mathscr{Q}}_t \xi \right \rangle  -\mathbb{E} \left[ \left\langle \ptr_{t,x}^{-1}\eta (X_t(x)),  \mathscr{Q}_t\int_{0}^{t} \mathscr{Q}_s^{-1} H_{\ptr_{s,x}}^{(k),\tr}\tilde{\mathscr{Q}}_s \xi \,  \d s\right\rangle\right],
\end{align}
 holds for $\eta \in \OO_0^{(k)}$.
Note that under condition {\bf (C)}, it follows from \eqref{BE-2} and the estimates \eqref{KKestimate1}, \eqref{esti-KK} that there exists a constant $A > 0$ such that
\begin{align}\label{GradEstimate}
|\nabla P_t^{(k)}\eta| \leqslant C \e^{A t} \left( (P_t |\nabla \eta|^2)^{1/2} +   (P_t |\eta|^2)^{1/2} \right),\quad\eta \in \OO_0^{(k)}.
\end{align}
It remains to show that estimate \eqref{GradEstimate} extends from
$\OO_0^{(k)}$ to $\OO_{b,1}^{(k)}$. This can be done by a standard approximation
argument. As $M$ is geodesically complete, there exists a sequence $(\varphi_n)_{n\in\N}$ of first order cut-off functions (e.g.~\cite[Theorem III.3~a)]{Gueneysu-Book}) with the properties
\begin{enumerate}[(i)]\openup-2\jot
\item $0\leqslant\varphi_n\leqslant1$ for all $n\in\N_+$;
\item for each compact $K\subset M$ there is $n_0(K)\in\N_+$ such that
$\varphi_n|K\equiv1$ for all $n\geqslant n_0(K)$;
\item $\|\nabla\varphi_n\|_\infty\to0$ as $n\to\infty$.
\end{enumerate}
We replace $\eta$ by $\eta_n:=\varphi_n\eta$ and then pass to the limit in the estimate as $n\to\infty$. From the local Bismut formula \eqref{localBismut} it is then easy to see that
$\nabla P_t^{(k)}\eta_n\to \nabla P_t^{(k)}\eta$ as $n\to\infty$. For the right-hand-side, we trivially have 
\(
(P_t|\eta_n|^2)^{1/2}
\to
(P_t|\eta|^2)^{1/2},
\)
and
\(
(P_t|\nabla\eta_n|^2)^{1/2}
\to
(P_t|\nabla\eta|^2)^{1/2}.
\)
Passing to the limit in the estimates for \(\eta_n\), one obtains the same estimates for
\(\eta\in\Omega_{b,1}^{(k)}\), as $n\to\infty$.
\end{proof}

\begin{remark}
  Since the estimates \eqref{EstUniform} are uniform on compact time intervals,
  it also follows that \eqref{MartCrucialnew} is a true martingale
for  any $\eta \in \OO_{b,1}^{(k)}$ and $\ell\in C^1([0,t])$, establishing the following global version of Bismut's formula:
\begin{align}\label{BE-1}
&\left\langle \nabla P_t^{(k)}\eta, \xi  \right\rangle(x)=-\mathbb{E} \left[ \left\langle  \ptr_{t, \, x}^{-1}\eta (X_{t}(x)), \mathscr{Q}_{t}\int_{0}^{t} \dot\ell_s \mathscr{Q}_s^{-1} \tilde{\mathscr{Q}}_s\xi  \, \d B_s + \mathscr{Q}_{t}\int_{0}^{t} \ell_s\mathscr{Q}_s^{-1} H _{\ptr_{s,x}}^{(k),\tr}\tilde{\mathscr{Q}}_s\xi   \,  \d s\right\rangle\right],
\end{align}
for a general deterministic  $\ell\in C^1([0,t])$ with  $ \ell_t=0 $  and $ \ell_0 =1$ as well.
A standard choice for $\ell_s$ is $\ell_s:=(t-s)/t$, so that $\dot\ell_s=-1/t$.
\end{remark}

\bibliographystyle{plain}
\providecommand{\bysame}{\leavevmode\hbox to3em{\hrulefill}\thinspace}
\providecommand{\MR}{\relax\ifhmode\unskip\space\fi MR }

\end{document}